\documentclass[11pt, a4paper, oneside, reqno]{amsart}

\newtheorem{theorem}{Theorem}[section]

\newtheorem{lemma}[theorem]{Lemma}
\newtheorem{assumption}[theorem]{Assumption}
\newtheorem{remark}[theorem]{Remark}

\newtheorem{corollary}[theorem]{Corollary}

\usepackage{comment}
\renewcommand{\epsilon}{\varepsilon}
\newcommand{\blank}[1]{\hspace*{#1}\linebreak[0]}

\newenvironment{myproof}[1] {{\it #1.}}{\hfill\qed}

\makeatletter
\def\@settitle{\begin{center}%
		\baselineskip14\p@\relax
		\normalfont\LARGE\scshape\bfseries
		\@title
	\end{center}%
}
\makeatother

\makeatletter

\def\subsection{\@startsection{subsection}{2}%
	\z@{.5\linespacing\@plus.7\linespacing}{.5\linespacing}%
	{\normalfont\large\bfseries}}

\def\subsubsection{\@startsection{subsubsection}{3}%
	\z@{.5\linespacing\@plus.7\linespacing}{.5\linespacing}%
	{\normalfont\itshape}}

\usepackage[colorlinks	= true,
			raiselinks	= true,
			linkcolor	= MidnightBlue,
			citecolor	= Mahogany,
			urlcolor	= ForestGreen,
			pdfauthor	= {Peyman Mohajerin Esfahani},
			pdftitle	= {},
			pdfkeywords	= {},   
			pdfsubject	= {},
			plainpages	= false]{hyperref}
\usepackage[usenames, dvipsnames]{color}
\usepackage{dsfont,amssymb,amsmath,subfig, graphicx,fancyhdr,mdframed}
\usepackage{amsfonts,dsfont,mathtools, mathrsfs,amsthm,wrapfig} 
\usepackage[]{siunitx}
\usepackage{ragged2e}
\usepackage{enumitem}

\allowdisplaybreaks
\usepackage{algorithm}
\usepackage{breqn}
\usepackage[font=small,labelfont=bf,
   justification=justified,
   format=plain]{caption}

\usepackage[noend]{algpseudocode}

\allowdisplaybreaks
\date{\today}

\patchcmd\@setauthors
  {\MakeUppercase{\authors}}
  {\authors}
  {}{}

\graphicspath{{images/}}

\title[Linesearch Method for Accelerated Composite Minimization]{Linesearch Method for \\ Accelerated Composite Minimization}

\author[R. {Rahimi Baghbadorani}]{Reza {Rahimi Baghbadorani}\textsuperscript{1}}
\author[S. Grammatico]{Sergio Grammatico\textsuperscript{1}}
\author[P. {Mohajerin Esfahani}]{Peyman {Mohajerin Esfahani}\textsuperscript{1,2}}

\thanks{{The authors are with (1) Delft Center for Systems and Control, Delft University of Technology, and (2) Operations Research, University of Toronto. This work was supported by the ERC grant TRUST-949796 and the NSERC Discovery grant RGPIN-2025-06544.}}

\thanks{{The authors are with (1) Rotterdam School of Management, Erasmus University, (2) Delft Center for Systems and Control, Delft University of Technology, and (3) Operations Research, University of Toronto. This work was supported by the ERC grant TRUST-949796 and the NSERC Discovery grant RGPIN-2025-06544.}}

\begin{document}

\maketitle

\begin{abstract}   
The choice of the stepsize in first-order convex optimization is typically based on the smoothness constant and plays a crucial role in the performance of algorithms. Recently, there has been a resurgent interest in introducing adaptive stepsizes that do not explicitly depend on the smooth constant. In this paper, we propose a novel linesearch stepsize rule based on function evaluations (i.e., zero-order information) that enjoys provable convergence guarantees for both accelerated and non-accelerated gradient descent. We further discuss the similarities and differences between the proposed stepsize regimes and the existing stepsize rules (including Polyak and Armijo). We numerically benchmark the performance of our proposed algorithms against state-of-the-art methods across three major problem classes of (1)~smooth minimization (logistic regression, quadratic programs, log-sum-exponential, and smooth max-cut relaxation) (2)~composite minimization ($\ell_1$-regularized least-squares, $\ell_1$-constrained least-squares, and $\ell_1$-regularized logistic regression), and (3)~non-convex minimization (cubic minimization). These classes include a wide range of operations research and management applications such as portfolio optimization, discrete choice models, regression, sparse classification and feature selection, high-order optimization and trust-region subproblems.

\end{abstract}


\section{Introduction}\label{intro}
Thanks to its simple implementation, scalability, and applicability, gradient descent (GD) is arguably the most popular algorithm to efficiently solve large-scale optimization problems, which finds a wide range of applications in operations research \cite{im2025stochastic, xu2021first, ho2018online}, machine learning \cite{Battiti1992, lan2020first}, control and system identification~\cite{hardt2018gradient, fazel2013hankel}. The main challenge in using gradient descent is the choice of the right stepsize, which has a considerable impact on the convergence speed. Several works have studied this choice, yet either the fastest possible convergence rate cannot be guaranteed or an expensive linesearch at each iteration is needed \cite{Polyak1969, Malitsky2020, Nocedal2006, Drori2014}. In this paper, we address these issues by introducing a novel and simple linesearch method.

Let us consider the unconstrained convex minimization problem $\min\limits_{x \in \mathbb{R}^n} f(x)$, {where $f$ is a continuously differentiable, smooth, convex function}. The usual iterative algorithm is of the form
\begin{equation}\label{stepsize-search direction}
    x_{k+1} = x_k + \lambda_k d_k. 
\end{equation}
where $\lambda_k$ and $d_k$ represent the stepsize and the descent direction, respectively. In the last decades, many algorithms have been developed for choosing $\bigl( \lambda_k \bigr)_{k\in \mathbb{N}}$ and $\bigl( d_k \bigr)_{k\in \mathbb{N}}$ in \eqref{stepsize-search direction} either statically or adaptively. We review some classes of design choices in the following.

A key object playing an important role in determining most of the existing stepsize rules at the iteration $x_k$ is the objective function along the update direction $d_k = -\nabla f(x_k)$ defined by
\begin{equation}\label{phi func}
    \phi_k (\lambda) := f(x_k +\lambda d_k).
\end{equation}
When the function $f$ is convex, then the scalar function $\phi_k: \mathbb{R} \rightarrow \mathbb{R}$ defined in \eqref{phi func} is also convex. Several classical stepsize rules defined via $\phi_k$ in \eqref{phi func} are the following. 

\textbf{(i) GD with constant stepsize \cite{Nocedal2006}:} The simplest stepsize rule is the constant $\lambda_k = 1/{L}$, where $L$ is the so-called smoothness parameter (Lipschitz continuity of $\nabla f(x)$). The convergence rate of the suboptimality $f(x_k) - f(x^*)$ with the constant stepsize is $\mathcal{O}(k^{-1})$.

\textbf{(ii) GD with exact linesearch \cite{Nocedal2006}:} Another classical stepsize is to find the optimal stepsize minimizing the function $\phi_k$ along the direction $d_k$, $\lambda_k = \arg \min_{\lambda} \phi_k(\lambda)$, which involves a scalar optimization problem that requires evaluation of the original function $f$.

\textbf{(iii) GD with approximated stepsize \cite{Nocedal2006}:} Due to the complexity of finding the stepsize by solving the exact linesearch minimization, alternative stepsize rules have been proposed that only require ensuring an inequality such as $\phi_k(\lambda) \leq \phi_k(0) + c \phi_k '(0)$ for some predefined constant $c > 0$. Among others, Armijo \cite{Armijo1966}, Wolfe, and strong Wolfe linesearch \cite{Nocedal2006} are backtracking linesearch methods falling into this category. 
    \begin{subequations}\label{backtracking linesearch}
    \begin{align}
        & \text{Armijo}: \qquad\qquad \phi_k(\lambda) \leq \phi_k(0) + c_1 \lambda \phi '_k(0), \,\, c_1 \in (0,1). \label{Armijo-rule}\\
        & \text{Wolfe}: \qquad\qquad\quad \phi'_k(\lambda) \geq c_2 \phi'_k(0), \,\, c_2 \in (c_1, 1) \, + \,\text{Armijo condition}. \label{Wolfe-rule}\\
        & \text{strong Wolfe}: \qquad |\phi'_k(\lambda)| \leq c_2 |\phi'_k(0)|, \,\, c_2 \in (c_1, 1) + \,\text{Armijo condition}. \label{strong-Wolfe-rule}
    \end{align}
    \end{subequations}
The suboptimality convergence rate of these algorithms is $\mathcal{O}(k^{-1})$. There are also several stepsize rules approximating the smooth constant locally based on the zero-order oracle's information (i.e., the objective function evaluation) \cite{poliak1987introduction, Malitsky2020, barre2019polyak, barre2020complexity}. In the following, we review two important stepsize rules that fall into this category:

\textbf{(iv) GD with the Polyak stepsize~\cite{poliak1987introduction}:} One of the classic adaptive stepsize rules is the Polyak stepsize, which uses the zero-order oracle's information at each iteration to locally approximate the smooth constant of the objective function:
    \begin{align}\label{Polyak stepsize}
        \lambda_k = \dfrac{f(x_k) - f(x^*)}{\| \nabla f(x_k)\|^2}.
    \end{align}
Although the Polyak stepsize offers optimal convergence rates for GD \eqref{stepsize-search direction} when the update direction is set to $d_k = -\nabla f(x_k)$ \cite{hazan2019revisiting}, it cannot be applied in problems where the optimal function value $f(x^*)$ is not available. We will get back to the Polyak stepsize and its relation to our proposed rule in Subsection \ref{Non-accelerated smooth minimization} (see equations \eqref{G_interval}).
    
\textbf{(v) GD with adaptive stepsize \cite{Malitsky2020}:} Recently, the following adaptive stepsize has been proposed based on the idea of approximating the smoothness constant locally: 
        \begin{equation} \label{malitsky ad-stepsize}
            \lambda_k = \min \left \{\sqrt{1+\theta_{k-1}}\lambda_{k-1}, \, \dfrac{\|x_k - x_{k-1}\|}{2\|\nabla f(x_k) - \nabla f(x_{k-1})\|}\right \},\quad\theta_k = \dfrac{\lambda_k}{\lambda_{k-1}}.
        \end{equation}
The second term on the left-hand side of \eqref{malitsky ad-stepsize} represents this local approximation as the inverse of the smoothness constant, while the first term ensures a required convergence rate of this parameter. The theoretical sub-optimality convergence rate for the stepsize rule \eqref{malitsky ad-stepsize} is $\mathcal{O}(k^{-1})$, however, the adaptive (i.e., state-dependent) nature of \eqref{malitsky ad-stepsize} has the advantage of (a) lifting the knowledge of the smoothness constant for the objective function, and (b) improving the practical convergence rate when the local smoothness is smaller than the global constant one.
Next, we consider methods with a dynamic direction $d_k$ in \eqref{stepsize-search direction}. The update rules in this class of methods have an extra term, which is called momentum, and allows the descent direction to have some inertia in the search space \cite{Polyak1969}. 

\textbf{(vi) AGD with constant stepsize \cite{Nesterov1983}:}
        \begin{equation*}
        \begin{cases}
            \beta_k = \dfrac{1+\sqrt{1+4\beta_{k-1}^2}}{2}, \,\, \gamma_k = \dfrac{1-\beta_k}{\beta_{k+1}},\\
            \lambda_k = \dfrac{1}{L},\,\,\, d_k = -\nabla f(x_k)-\gamma_k \left (\dfrac{\lambda_{k-1}}{\lambda_k} d_{k-1} - \nabla f(x_k) + \dfrac{\lambda_{k-1}}{\lambda_k}\nabla f(x_{k-1}) \right),
        \end{cases}
        \end{equation*}
where $L$ is the smooth constant of $f$. The descent direction in this method uses previous iteration values to accelerate the convergence rate, which is $\mathcal{O}(k^{-2})$ for $f(x_k) - f(x^*)$.
    
\textbf{(vii) AGD with adaptive stepsize \cite{Malitsky2020}:} A heuristic version of adaptive stepsize \eqref{malitsky ad-stepsize} for accelerated methods is also proposed in \cite[Section 3]{Malitsky2020}
    \begin{equation*}
        \begin{cases}
            \lambda_k = \min \left \{\sqrt{1+\dfrac{\theta_{k-1}}{2}}\lambda_{k-1}, \, \dfrac{\|x_k - x_{k-1}\|}{2\|\nabla f(x_k) - \nabla f(x_{k-1})\|}\right \}, \,\,\, \theta_k = \dfrac{\lambda_k}{\lambda_{k-1}},\\
            \Lambda_k = \min \left \{\sqrt{1+\dfrac{\Theta_{k-1}}{2}}\Lambda_{k-1}, \, \dfrac{\|\nabla f(x_k) - \nabla f(x_{k-1})\|}{2\|x_k - x_{k-1}\|}\right \}, \,\,\, \Theta_k = \dfrac{\Lambda_k}{\Lambda_{k-1}},\\
            \beta_k = \dfrac{\sqrt{1/ \lambda_k} - \sqrt{\Lambda_k}}{\sqrt{1/ \lambda_k} + \sqrt{\Lambda_k}},\\
            d_k = \beta_k \left(\dfrac{\lambda_{k-1}}{\lambda_k} d_{k-1} - (1 + \dfrac{1}{\beta_k})\nabla f(x_k) + \dfrac{\lambda_{k-1}}{\lambda_k}\nabla f(x_{k-1}) \right).
        \end{cases}
    \end{equation*}
This adaptive rule presents promising numerical performance, while there is no formal theoretical guarantee explaining this interesting performance. This theoretical gap toward accelerated techniques is one of the motivations of this study.
        
We also consider the composite minimization problem
            \begin{equation}\label{composite problem}
                \min_{x \in \mathcal{X}} F(x) = \min_{x \in \mathcal{X}} f(x)+h(x)
            \end{equation}
where $f\colon \mathcal{X} \to \mathbb{R}$ is a convex and smooth function, and $h\colon \mathcal{X} \to \mathbb{R}$ is convex and possibly non-smooth but prox-friendly (technical details are postponed to Section \ref{Non-accelerated adaptive stepsize}). There are several methods for convex composite minimization in the existing literature. Among them, subgradient descent and mirror descent are two well-known classic methods that suffer from a slow convergence rate $\mathcal{O}(k^{-1/2})$ \cite{nemirovskij1983problem,Beck2003,yurtsever2018conditional,yang2018rsg}. Following the seminal works by Nesterov, one can exploit the structure of the nonsmooth part in the objective function and deploy a first-order accelerated method.  ISTA and FISTA \cite{Beck2009}, which are extensions of GD and NAGD, respectively, offer a faster convergence rate compared to subgradient and mirror descent. However, these methods require knowledge of the smooth constant associated with the smooth part, $f$ in \eqref{composite problem}. Additionally, there are various backtracking techniques available to approximate the smooth constant, but they often tend to be conservative, resulting in larger values than the actual smooth constant. This conservatism can impact the convergence speed. The authors in \cite{vladarean2021first, malitsky2020golden} propose an algorithm that enjoys the local smoothness of $f$ to determine the stepsize. However, despite its simplicity and efficacy, the convergence guarantee is still the suboptimal rate of $\mathcal{O}(k^{-1})$. More recently, the authors in \cite{li2023simple} propose a fully adaptive stepsize with momentum acceleration and a convergence rate of $\mathcal{O}(k^{-2})$. However, the algorithm requires several hyperparameters, which may affect the performance in practice. Inspired by this, we develop a \textcolor{black}{zeroth-order linesearch rule for gradient-type algorithms applied to} convex composite minimization problems \eqref{composite problem} with the optimal theoretical convergence rate $\mathcal{O}(k^{-2})$ while maintaining a reasonable level of computational complexity.
        
\textbf{Contributions.} Following the same spirit of the adaptive stepsize and linesearch rules, this study proposes a novel stepsize rule without the knowledge of the global smoothness constant for solving convex composite minimization problems. Using a Lyapunov-based argument, we show that the proposed rule enjoys the optimal worst-case complexity bound of $\mathcal{O}(k^{-1})$ and $\mathcal{O}(k^{-2})$ in both cases of nonaccelerated and accelerated algorithms, respectively. These results are developed for a general class of composite convex functions and are optimal in the sense that they match the theoretical lower bounds of the class of first-order algorithms for this class of functions. Our convergence results are summarized in Table \ref{table1} where our update direction $d_k = -G^{f}_{\lambda h}(x_k)$ is our gradient mapping (see the definition \eqref{grad_mapping} for more details concerning the gradient mapping). Let us note that the results in Table \ref{table1} are reported for the general composite case \eqref{composite problem}. If the nonsmooth term is $h(x) = 0$, then the gradient mapping reduces to $G^{f}_{\lambda h}(x) = \nabla f(x)$. In the rest of the introduction, we briefly discuss our proposed stepsize rule in this special case of a smooth convex function and provide a geometrical illustration in comparison with the existing rules reviewed above. 
        \begin{table}[H]
           \centering
           \caption{Summary of the main results of this paper.}
           \small
           \begin{tabular}{c c c c}\hline
         {Algorithm} & {Problem} & {Stepsize rule} & {Convergence rate}\\
           \hline & \\[-1.0em]
            nonaccelerated (Theorem \ref{theorem1}) & composite & $\phi(2\lambda) \leq \phi(\lambda) - \lambda \langle G^{f}_{\lambda h}(x), \nabla f(x) \rangle + \dfrac{\lambda}{2}\|G^{f}_{\lambda h}(x)\|^2$ & $\mathcal{O}(k^{-1})$\\
            nonaccelerated (Corollary \ref{corollary-smooth-nonaccelerated}) & smooth & $\phi(2\lambda) \leq \phi(\lambda) + \dfrac{\lambda}{2}\phi^\prime(0)$ & $\mathcal{O}(k^{-1})$\\
            accelerated (Theorem \ref{theorem3}) & composite & $\phi(2\lambda) \leq \phi(\lambda) - \lambda \langle G^{f}_{\lambda h}(x), \nabla f(x) \rangle + \dfrac{\lambda}{2}\|G^{f}_{\lambda h}(x)\|^2$ & $\mathcal{O}(k^{-2})$\\
            accelerated (Corollary \ref{remark for smooth}) & smooth & $\phi(2\lambda) \leq \phi(\lambda) + \dfrac{\lambda}{2}\phi^\prime(0)$ & $\mathcal{O}(k^{-2})$\\
            \\[-1.0em]
            \hline
           \end{tabular}
           \label{table1}
       \end{table}    
\textbf{Partial result}: If $f$ is a smooth and convex function, our proposed stepsize rule is defined as
        \begin{equation} \label{partial stepsize rule}
          \lambda_{k}^{\mathrm{our}} := \arg\max\limits_{\lambda} \left \{\lambda \,|\,\,  \phi_k(2\lambda) \leq \dfrac{1}{2}\phi_k^{\prime}(0)\lambda + \phi_k(\lambda) \right\}
        \end{equation}
        where the function $\phi_k$ is defined in \eqref{phi func} and visualized in Figure \ref{fig:phi}. Recall that in the case of non-accelerated methods, the update direction is simply $d_k = -\nabla f(x)$. We emphasize that the proposed linesearch above only requires additional zeroth-order information, namely the objective function evaluation, while the first-order information is the gradient of the function, which is already computed in the previous step. Figure \ref{fig:linesearch methods} illustrates a geometric interpretation of the proposed stepsize rule \eqref{partial stepsize rule}, together with the existing literature reviewed in the preceding section. For further geometrical interpretation, particularly in the case of accelerated methods, we refer to \cite{bubeck2015geometric}.

\begin{figure}[ht]
    \centering
    \captionsetup{justification=centering}
    \subfloat[$\phi_k(\lambda)$.]{\label{fig:phi}\includegraphics[scale=0.75]{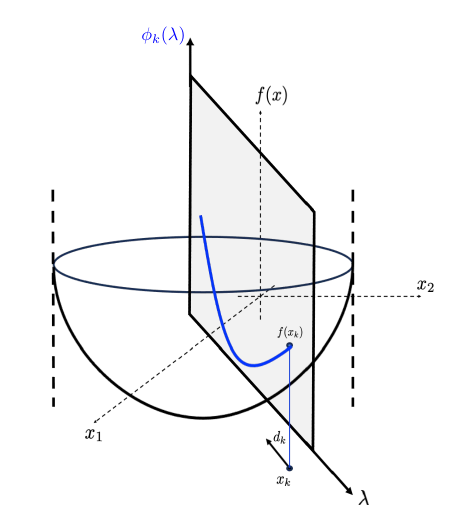}}
        \hspace{1em}
    \subfloat[Stepsize rules.]{\label{fig:linesearch methods}\includegraphics[scale=0.40]{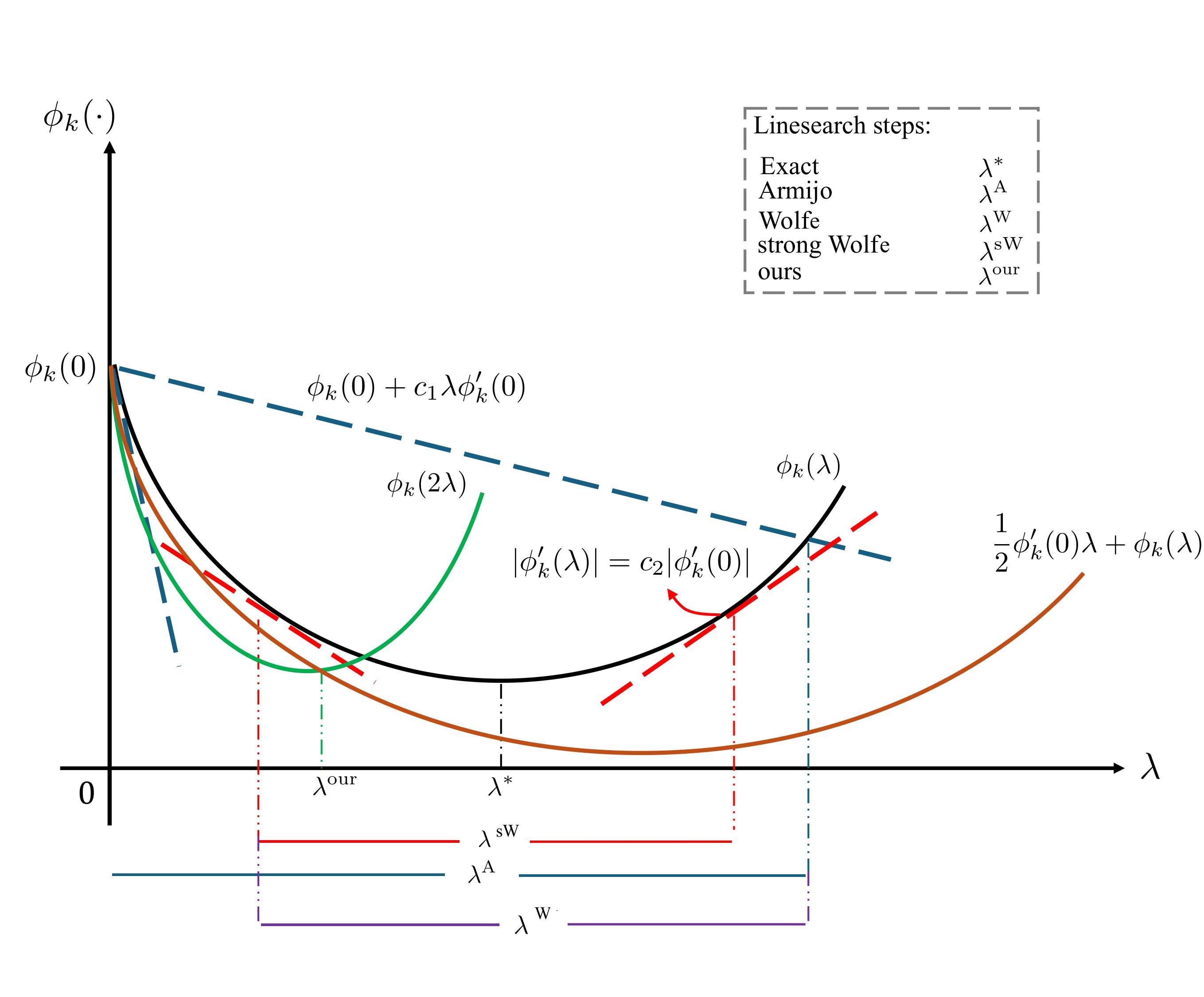}}
    \caption{Geometric interpretation of different stepsize rules using $\phi_k(\lambda)$ defined in \eqref{phi func}.}
    \label{fig1}
\end{figure}
The paper is organized as follows: In Section \ref{Non-accelerated adaptive stepsize}, we present the theoretical results for the non-accelerated adaptive stepsize algorithm. Section \ref{Accelerated adaptive stepsize} provides technical proofs for the accelerated adaptive stepsize case. Some implementation techniques and illustrative examples to show the efficiency of our approaches are presented in Section \ref{Numerical results}. Finally, the conclusion and further discussion are given in Section \ref{conclussion}.

\textbf{Notation.} Let $\mathcal{X}$ be the finite-dimensional real vector space with the standard inner product $\langle \cdot, \cdot \rangle$ and $\ell_p$-norm  $\|\cdot\|_p$ (by $\|\cdot\|$, we mean the Euclidean standard 2-norm). If $f$ is differentiable, then $\nabla f(x)$ represents the gradient of $f$ at $x$. The function $f$ is $L$-smooth, or equivalently the gradient of $f$ is $L$-Lipschitz, if for all $x,y \in \mathbb{R}^d$ one of the following inequalities is satisfied: 
\begin{subequations}\label{smoothness inequality}
\begin{align}
    & \|\nabla f(x) - \nabla f(y)\| \leq L\|x-y\|, \label{eq2} \\
    & f(y) \leq f(x) + \langle \nabla f(x),y-x \rangle + \dfrac{L}{2} \|x-y\|^2. \label{eq1}
\end{align}
\end{subequations}
Furthermore, $f$ is locally smooth if it is smooth over any compact set of its domain, see \cite[Chap. 2]{Nesterov2018} for more details. The prox operator of a convex function $h \colon \mathcal{X} \to \mathbb{R}$ is defined as
\begin{equation}\label{prox-operator}
    \text{prox}_{h}(x) = \arg \min_{u} h(u) + \dfrac{1}{2} \|u-x\|^2.
\end{equation}
A function is ``prox-friendly'' if the prox operator in~\eqref{prox-operator} is (computationally or explicitly) available. We also denote the gradient mapping of two convex functions~$f$ and $h$ by 
\begin{align}
\label{grad_mapping}
G^{f}_{\lambda h}(x) := \dfrac{1}{\lambda}\Big(x - \mathrm{prox}_{\lambda  h}\big(x - \lambda \nabla f(x)\big)\Big),   
\end{align}
where $\lambda$ is a positive scalar and has a stepsize interpretation. The gradient mapping is available if $f$ is differentiable and $h$ is prox-friendly.
\section{Non-Accelerated Adaptive Stepsize} \label{Non-accelerated adaptive stepsize}
We present our algorithm and its convergence analysis for the non-accelerated case in this section. From a high-level viewpoint, the proposed rule follows the proximal gradient descent method with the difference in the choice of stepsize $\lambda$, which depends on the previous state and its gradient.
\subsection{Preliminaries}\label{Preliminary}
 We first proceed with some assumptions and lemmas that will be used throughout the paper. We note that the classical Cauchy-Schwarz and convexity inequalities are two useful tools in our analysis, and most of the results build on the seminal works by Nesterov \cite{Nesterov1983} and Polyak \cite{Polyak1969}. The following assumptions hold throughout this study. 
 
\begin{assumption}[Convex regularity]\label{assum2}
The convex function~$F$ in \eqref{composite problem} has bounded level sets and a finite minimizer, where the convex terms~$f(x)$ and $h(x)$ are smooth and prox-friendly, respectively. 
\end{assumption}
The next lemma indicates several properties of the composite minimization \eqref{composite problem} that are central to develop the algorithms in this paper. 
\begin{lemma}[Gradient mapping]\label{keylemma}
Let $G^{f}_{\lambda h}(x)$ be the gradient mapping defined in \eqref{grad_mapping} for a smooth convex function $f$, a possibly nonsmooth function $h$, and a positive constant $\lambda$ in $\mathbb{R}_+$.
\begin{enumerate}[label=(\roman*), itemsep = 0mm, topsep = 0mm, leftmargin = 7mm]
\item \label{lema3}
{\bf Stationary condition:} The vector $x$ is a minimizer of \eqref{composite problem} if and only if $G^{f}_{\lambda h}(x) = 0$.

\item \label{lema1}
{\bf Convexity-like inequality:} For all $x,y$ in $\mathbb{R}^d$ we have
\begin{align*}
    h\left(x-\lambda G^{f}_{\lambda h}(x)\right) \leq h(y) - \langle G^{f}_{\lambda h}(x) - \nabla f(x), y - (x - \lambda G^{f}_{\lambda h}(x)) \rangle.  
\end{align*}

\item \label{lema2}
{\bf Increment bound:} Considering $x^+ = x - \lambda G^{f}_{\lambda h}(x)$, for any $z \in \mathbb{R}^d$ we have
\begin{equation*}\label{eq3}
F(x^+) - F(z) \leq \langle x^+ - x,  \nabla f(x^+) - \nabla f(x) + \dfrac{1}{2}G^{f}_{\lambda h}(x) \rangle - \Big( \dfrac{1}{2\lambda} \|x^+ - x\|^2 + \dfrac{1}{\lambda}\langle x^+ -x, x - z \rangle \Big).
\end{equation*}

\item \label{lema4linesearch}
{\bf Zero-order linesearch:} Let the update direction~$d_k = -G^{f}_{\lambda h}(x)$. If $\lambda$ satisfies
\begin{equation}\label{ZO linesearch}
    \phi_k(2\lambda) \leq \phi_k(\lambda) - \lambda \langle G^{f}_{\lambda h}(x), \nabla f(x) \rangle + \dfrac{\lambda}{2}\|G^{f}_{\lambda h}(x)\|^2,
\end{equation}
we then have $\langle x^+ - x,  \nabla f(x^+) - \nabla f(x) + \dfrac{1}{2}G^{f}_{\lambda h}(x) \rangle \leq 0$.
\end{enumerate}
\end{lemma}
Before providing the technical proof of the lemma, let us offer some insights into why the above four statements will help us develop algorithms: 
The convex-like inequality \ref{lema1} is a gradient mapping intrinsic characteristic that is particularly useful for controlling the increment of the original function $F$ in \eqref{composite problem} along the direction of $d_k = -G^{f}_{\lambda h}(x)$. In both cases of the non-acceleration and acceleration approaches, the increment inequality \ref{lema2} is crucial in proving our convergence analysis. The increment bound~\ref{lema2} allows for the inclusion of a momentum term that emerges in the acceleration dynamics. The first term on the right-hand side of \ref{lema2} is an unfavorable term which is often targeted via the stepsize rule to remain negative. Checking the linesearch in \ref{lema4linesearch} guarantees the negativity of this unfavorable term, which only requires additional zeroth-order information of function $F$ while the gradient (first-order) information of the gradient mapping $G^{f}_{\lambda h}(x)$ has already been computed. This is a crucial feature, making a stepsize rule computationally more useful in practice. We conclude this section with the proof of the lemma.

\begin{myproof}{Proof of Lemma \ref{keylemma}}
We provide the proof of each part separately as follows:\\
Part \ref{lema3}: Assume $\hat{x}$ is the minimizer of the composite minimization \eqref{composite problem}. Then, by the definition of \text{prox} operator \eqref{prox-operator}, we have the following equivalent implications
\begin{align*}
G^{f}_{\lambda h}(\hat{x}) = 0 & \iff \hat{x} = \text{prox}_{\lambda h}\left(\hat{x} - \lambda \nabla f(\hat{x})\right) \iff (\hat{x} - \lambda \nabla f(\hat{x})) - \hat{x} \in \lambda \partial h(\hat{x}) \\
& \iff - \lambda \nabla f(\hat{x}) \in \lambda \partial h(\hat{x}) \iff 0 \in \nabla f(\hat{x}) + \partial h(\hat{x}) \iff \hat{x} \,\, \text{is the minimizer of \eqref{composite problem}}. 
\end{align*}
Part \ref{lema1}: Let $r = \text{prox}_{\lambda h}(t)$. By the definition \eqref{prox-operator} and the first order optimality condition,  we can write
\begin{equation} \label{proof lemma sec2}
r = \arg \min_{r} h(r) + \dfrac{1}{2\lambda} ||r-t||^2 \iff  0 \in \partial{h(r)}+\dfrac{1}{\lambda}(r-t)\iff t-r \in \lambda \partial{h(r)}   .
\end{equation}
By defining $u := x-\lambda G^{f}_{\lambda h}(x)$ and $w := x-\lambda \nabla f(x)$ we have
\begin{equation*}
u = x-\lambda G^{f}_{\lambda h}(x) = x - \lambda \dfrac{1}{\lambda}\Big(x - \text{prox}_{\lambda h}\big(x - \lambda \nabla f(x)\big)\Big) = \text{prox}_{\lambda h}(x - \lambda \nabla f(x)) =  \text{prox}_{\lambda h}(w).
\end{equation*}
Now, using \eqref{proof lemma sec2}, we can conclude: $G^{f}_{\lambda h}(x) - \nabla f(x) \in \partial{h(x-\lambda G^{f}_{\lambda h}(x))}$.
Finally, using convexity of $h$ we can write
\begin{equation*}
h(x-\lambda G^{f}_{\lambda h}(x)) \leq h(y) - \langle G^{f}_{\lambda h}(x) - \nabla f(x), y - (x - \lambda G^{f}_{\lambda h}(x)) \rangle.  
\end{equation*}
\\
Part \ref{lema2}: By using the definition of \(F\) we have
\begin{align*}
  F(x^+) - F(z) &= f(x^+) + h(x^+) - f(z) - h(z) \nonumber \\
  &= f(x - \lambda G^{f}_{\lambda h}(x)) - f(z) + h(x - \lambda G^{f}_{\lambda h}(x)) - h(z) 
\end{align*}  
 Using \ref{lema1} on the right-hand-side of the above equality and the convexity of \(f\) and $h$ yields 
\begin{align}\label{aux-lemma2}
  F(x^+) - F(z) &\leq f(x - \lambda G^{f}_{\lambda h}(x)) - f(x) + \langle \nabla  f(x), x - z \rangle - \langle G^{f}_{\lambda h}(x) - \nabla f(x), z - (x - \lambda G^{f}_{\lambda h}(x)) \rangle \nonumber \\
  &\leq \langle \nabla  f(x - \lambda G^{f}_{\lambda h}(x)), x - \lambda G^{f}_{\lambda h}(x) - x \rangle + \langle G^{f}_{\lambda h}(x), x^+ - z\rangle + \langle \nabla f(x), \lambda G^{f}_{\lambda h}(x) \rangle \nonumber \\
  &= \langle x^+ - x, \nabla f(x^+) \rangle - \dfrac{1}{\lambda} \langle x^+ - x, x^+ - z \rangle + \langle x - x^+, \nabla f(x) \rangle.
  \end{align}
By adding and subtracting $\dfrac{1}{\lambda} \langle x^+ - x, x \rangle$ in \eqref{aux-lemma2}, we obtain
\begin{align*}
  & F(x^+) - F(z) \leq \langle x^+ - x, \nabla f(x^+) - \nabla f(x) + \dfrac{1}{2}G^{f}_{\lambda h}(x) \rangle - \left( \dfrac{1}{2\lambda}\|x^+ - x\|^2 + \dfrac{1}{\lambda} \langle x^+ - x, x - z \rangle \right).
\end{align*}\\
Part \ref{lema4linesearch}: Replacing the definition $\phi_k(\cdot)$ from \eqref{phi func} in \eqref{ZO linesearch} leads to
    \begin{align*}
        &f(x - 2\lambda G^{f}_{\lambda h}(x)) - f(x - \lambda G^{f}_{\lambda h}(x)) \leq -\lambda \langle G^{f}_{\lambda h}(x), \nabla f(x) \rangle + \dfrac{\lambda}{2}\|G^{f}_{\lambda h}(x)\|^2.
    \end{align*}
    Using the convexity of $f$ on the left-hand side of the above inequality yields
    \begin{align*}   
        & \langle \nabla f(x - \lambda G^{f}_{\lambda h}(x)), x - 2\lambda G^{f}_{\lambda h}(x) - x + \lambda G^{f}_{\lambda h}(x) \rangle \leq -\lambda \langle G^{f}_{\lambda h}(x), \nabla f(x) \rangle + \dfrac{\lambda}{2}\|G^{f}_{\lambda h}(x)\|^2.
    \end{align*}
    Considering the definition of $x^+$ in part \ref{lema2}, we rearrange the above inequality and arrive at
    \begin{align*}
        & \langle \nabla f(x - \lambda G^{f}_{\lambda h}(x)) - \nabla f(x) + \dfrac{1}{2} G^{f}_{\lambda h}(x), G^{f}_{\lambda h}(x)\rangle \geq 0  \\
        &\, \Longrightarrow \, \langle \nabla f(x^+) - \nabla f(x) + \dfrac{1}{2}G^{f}_{\lambda h}(x), x^+ - x \rangle \leq 0.
    \end{align*}
\end{myproof}
\subsection{Non-accelerated composite minimization}
This section focuses on devising an adaptive stepsize rule for non-accelerated algorithms (i.e., $d_k = - G^{f}_{\lambda h}(x)$) for the general class of composite functions \eqref{composite problem}. Next, we propose our non-accelerated stepsize rule for the general composite minimization problem~\eqref{composite problem}. In this case, the zero-order linesearch~\ref{lema4linesearch} is the driving force behind the proposed stepsize rule, after which we only need to apply the classic gradient descent update \eqref{stepsize-search direction} with $d_k = - G^{f}_{\lambda h}(x)$. This discussion is formalized in the next theorem. 
\begin{theorem}[Non-accelerated adaptive stepsize]\label{theorem1}
Consider the function~$F$ in \eqref{composite problem} as a composition of a smooth function $f(x)$ and a prox-friendly function $h(x)$. Assume that the sequence $\left (x_k \right)_{k \in \mathbb{N}}$ generated by the update algorithm \eqref{stepsize-search direction} in which the direction and stepsize rule are
\begin{equation*}\label{alg:Algorithm1}
\tag{Alg.\,1}
\begin{cases}
  \lambda_k = \max\limits \left \{\lambda \in \mathbb{R}_+ \,|\,\, \phi_k(2\lambda) \leq \phi_k(\lambda) - \lambda \langle G^{f}_{\lambda h}(x_k), \nabla f(x_k) \rangle + \dfrac{\lambda}{2}\|G^{f}_{\lambda h}(x_k)\|^2\right\},\\      
  d_k = -G^{f}_{\lambda_k h}(x_k).\\
\end{cases}
\end{equation*}
Then, we have the uniform stepsize lower bound $\lambda_k \geq {1}/{3L}$ where $L$ is the smoothness constant of $f$. Moreover, we have the sub-optimality bound $F(x_k) - F(x^*) \leq \dfrac{D}{k}$ where $D$ is a constant that only depends on the initial condition $x_0$, the optimal solution $x^*$, and the smoothness constant $L$.
\end{theorem}
\begin{proof}
First, to show that $\lambda_k$ is bounded from below by ${1}/{3L}$, we just need to show that the condition in the stepsize rule always is satisfied by $\lambda_k = 1/{3L}$. Considering the smoothness of $f$, we can write the inequality \eqref{eq1} for the pair $(x,x^+)$ where $x^+ = x - \lambda G^{f}_{\lambda h}(x)$, which arrives at the two inequalities
\begin{align*}
f(x^+) &\leq f(x) - \langle \nabla f(x),\lambda G^{f}_{\lambda h}(x) \rangle + \dfrac{L\lambda ^2}{2} \|G^{f}_{\lambda h}(x)\|^2, \\
f(x) &\leq f(x^+) + \langle \nabla f(x^+),\lambda G^{f}_{\lambda h}(x) \rangle + \dfrac{L\lambda ^2}{2} \|G^{f}_{\lambda h}(x)\|^2.
\end{align*}
Adding the two sides of the above inequalities leads to
\begin{equation}\label{eq4}
\langle G^{f}_{\lambda h}(x), \nabla f(x^+) - \nabla f(x) + {L\lambda}G^{f}_{\lambda h}(x) \rangle \geq 0, \qquad \forall \, \lambda \in \mathbb{R}_+, \, x \in \mathbb{R}^n.   
\end{equation}
{By the smoothness of the scalar function $\phi_k(\lambda) = f(x_k + \lambda d_k)$, it follows that
\begin{align*}
    f(x_k + 2\lambda d_k)
    \leq f(x_k + \lambda d_k)
    + \nabla f(x_k + \lambda d_k)^\top \bigl( (x_k + 2\lambda d_k) - (x_k + \lambda d_k) \bigr)
    + \frac{L}{2}\|\lambda d_k\|^2.
\end{align*}
Rewriting the above inequality yields
\begin{align*}
    f(x_k + 2\lambda d_k) - f(x_k + \lambda d_k)
    \leq \lambda\, d_k^\top \nabla f(x_k + \lambda d_k)
    + \frac{L\lambda^2}{2}\|d_k\|^2.
\end{align*}
To ensure that the stepsize rule in \ref{alg:Algorithm1} is satisfied, and using the definition of the direction $d_k$, we require
{\color{black}
\begin{align*}
    - \lambda G^{f}_{\lambda h}(x_k)^\top
    \nabla f\bigl(x_k - \lambda G^{f}_{\lambda h}(x_k)\bigr)
    + \frac{L\lambda^2}{2}\|G^{f}_{\lambda h}(x_k)\|^2
    \leq
    - \lambda \langle G^{f}_{\lambda h}(x_k), \nabla f(x_k) \rangle
    + \frac{\lambda}{2}\|G^{f}_{\lambda h}(x_k)\|^2.
\end{align*}
Rearranging the terms and dividing both sides by $\lambda>0$, this inequality is equivalent to}
\begin{align}\label{thr2_3_ineq1}
    0 \leq
    \left\langle G^{f}_{\lambda h}(x_k),
    \nabla f\bigl(x_k - \lambda G^{f}_{\lambda h}(x_k)\bigr) - \nabla f(x_k)
    \right\rangle
    + \frac{1 - \lambda L}{2}\|G^{f}_{\lambda h}(x_k)\|^2.
\end{align}
Comparing this expression with inequality~\eqref{eq4}, which holds universally, we deduce that the inequality \eqref{thr2_3_ineq1} holds whenever
\begin{align}\label{thr2_3_ineq2}
    \lambda L \leq \frac{1 - \lambda L}{2},
\end{align}
which implies $\lambda \leq \frac{1}{3L}$.}
Next, we prove the convergence of \ref{alg:Algorithm1}. By putting $z = x_k$ in the inequality \ref{eq3} we have
\begin{equation}\label{conv-analysis}
F(x_{k+1}) - F(x_k) \leq - \dfrac{\lambda_k}{2} \|G^{f}_{\lambda_k h}(x_k)\|^2  
\end{equation}
\textcolor{black}{Summing the previous inequality from $k=0$ to $T-1$ and utilizing the stepsize lower bound $\lambda_k \geq 1/(3L)$, we obtain
\begin{equation*}
\frac{1}{6L}\sum_{k=0}^{T-1} \|G^{f}_{\lambda_k h}(x_k)\|^2 \leq F(x_0) - F(x_{T}).
\end{equation*}
By Assumption \ref{assum2}, the objective function $F(x)$ is bounded from below by its optimal value $F^*$. Substituting this lower bound and taking the limit as $T \to \infty$ yields
\begin{equation*}
\sum_{k=0}^\infty \|G^{f}_{\lambda_k h}(x_k)\|^2 \leq 6L (F(x_0) - F^*) < \infty.
\end{equation*}
Because the sequence of function values is monotonically decreasing and bounded from below, the non-negative infinite series on the left-hand side converges to a finite value. This implies that the elements of the series must asymptotically approach zero. Thus, we conclude that
\begin{equation*}
\lim_{k \to \infty} \|G^{f}_{\lambda_k h}(x_k)\|^2 = 0,
\end{equation*}
which satisfies the standard first-order optimality conditions and establishes convergence by Lemma~\ref{keylemma}~\ref{lema3}.}
To derive the rate of convergence we start from \ref{lema2} with $z = x^*$.
\begin{equation}\label{aux-conv}
   F(x_{k+1}) \leq F(x^*) - \dfrac{1}{2\lambda_k} \|x_{k+1} - x_k\|^2 - \dfrac{1}{\lambda_k}\langle x_{k+1} - x_k, x_k -  x^* \rangle  
\end{equation}
By adding and subtracting $\dfrac{1}{\lambda_k} \langle x_{k+1}, x_k \rangle$ and $\dfrac{1}{\lambda_k} \langle x_k, x_k \rangle$ terms and expanding the right-hand side of \eqref{aux-conv} we have  
\begin{equation}\label{aux-conv2}
     F(x_{k+1}) \leq F(x^*) + \dfrac{1}{2\lambda_k}\Big(\|x_k - x^*\|^2 - \|x_{k+1} - x^*\|^2\Big)
\end{equation}
{\color{black} Using \eqref{aux-conv2} and the lower bound of $\lambda_k$, we can write
\begin{equation*}
    F(x_{k+1}) \leq F(x^*) + \frac{3}{2}L\Big(\|x_k - x^*\|^2 - \|x_{k+1} - x^*\|^2\Big)
\end{equation*}
Summing the above inequality from $k = 0$ to $k = T-1$ gives us
\begin{equation*}
    \sum_{k=1}^{T}\Big(F(x_k) - F(x^*)\Big) \leq \frac{3}{2}L\Big(\|x_0 - x^*\|^2 - \|x_T - x^*\|^2\Big) \leq \frac{3}{2}L\|x_0 - x^*\|^2
\end{equation*}
Because the sequence of $F(x_k)$ is nonincreasing, averaging yields
\begin{align*}
    F(x_T) - F(x^*) \leq \dfrac{1}{T}\sum_{k=1}^{T}\Big(F(x_k) - F(x^*)\Big) \leq \frac{3}{2T}L\Big(\|x_0 - x^*\|^2 - \|x_T - x^*\|^2\Big) \leq \dfrac{3L\|x_0 - x^*\|^2}{2T} = \dfrac{D}{T}.
\end{align*}}
\end{proof}
Theorem \ref{theorem1} can be extended to a slightly more general setting where $f$ is locally smooth.
\begin{corollary}[Locally smooth function]
The sequence $\left (x_k \right)_{k \in \mathbb{N}}$ generated by \eqref{stepsize-search direction} for the direction and stepsize rule \ref{alg:Algorithm1} ensures the boundedness of the sequence and the sub-optimality bound $F(x_k) - F(x^*) \leq \dfrac{D}{k}$, where $f$ is locally smooth function in \eqref{composite problem}.
\end{corollary}
\begin{proof}
This extension essentially indicates that the smoothness condition concerning the term $f$ reflected in the constant $L$ is only required over a compact set. To this end, we recall that the proposed algorithm is monotone thanks to the inequality \eqref{conv-analysis}. That is, $F(x_{k+1}) \leq F(x_k)$ for all $k$, implying that the entire sequence of $\left (x_k \right)_{k \in \mathbb{N}}$ remains in the sublevel set of $\{x \in \mathbb{R}^n ~:~ F(x) \le F(x_0)\}$. Note also that the regularity condition in Assumption \ref{assum2} implies that the levelset of the objective function $F$ is compact. This concludes the desired assertion. 
\end{proof}
\begin{remark}[Approximate adaptive stepsize rule] \label{stepsize appr}
    The stepsize rule \ref{alg:Algorithm1} involves a scalar optimization problem. In practice, we use a simple backtracking method to approximate the optimal solution. To this end, we start with an initial $\lambda_0$ and scale it with a coefficient $C \in (0,1)$ as 
\begin{equation}\label{back-line}
    \lambda_k = \max \left \{\lambda_0 C^i \,|\,\, \phi_k(2\lambda) \leq \phi_k(\lambda) - \lambda \langle G^{f}_{\lambda h}(x_k), \nabla f(x_k) \rangle + \dfrac{\lambda}{2}\|G^{f}_{\lambda h}(x_k)\|^2, \,\, i \in \mathbb{N}\right\}.
\end{equation}
The above approximate solution is guaranteed to satisfy $\lambda_k \geq \dfrac{C}{3L}$.
\end{remark}
{\color{black} We would like to conclude this subsection by discussing the structure of the proposed stepsize rule in~\ref{alg:Algorithm1} and its possible generalization in the following remark.
\begin{remark}[Generalization of the linesearch rule]\label{Generalization of the linesearch rule}
The factor $(2)$ in $\phi_k(2\lambda)$ is dictated by the structure of the increment bound in Lemma~2.2~(iii)--(iv), which is used in the proof of Theorem~2.3. In particular, replacing $2$ by an arbitrary $\alpha$ does not, in general, yield the corresponding increment bound required for the subsequent analysis. More specifically, the increment bound in Lemma~2.2~(iii)--(iv) and the relations in \eqref{eq4}--\eqref{thr2_3_ineq1} must hold simultaneously, which cannot, in general, be guaranteed when the factor $2$ is replaced by an arbitrary $\alpha$. Thus, this choice is structural rather than merely a consequence of proof convenience. Before continuing, we note that, in the smooth setting, the above restriction can nevertheless be relaxed, leading to a more general result as discussed in Subsection~\ref{Non-accelerated smooth minimization}, specifically in Remark~\ref{smooth-Generalization of the linesearch rule}.

The coefficient of the quadratic term can, however, be generalized. For any $\rho\in(0,2)$, as required to obtain the result in Lemma~2.2~(iv), consider
\begin{align*}
\phi_k(2\lambda)
\leq
\phi_k(\lambda)
-\lambda
\langle
G^f_{\lambda h}(x_k),\nabla f(x_k)
\rangle
+\frac{\rho\lambda}{2}
\|G^f_{\lambda h}(x_k)\|^2.
\end{align*}
Following the same arguments as in Theorem~2.3 yields
\begin{align*}
F(x_{k+1})-F(x_k)
\leq
\frac{(\rho-2)\lambda_k}{2}
\|G^f_{\lambda_k h}(x_k)\|^2,
\end{align*}
while the guaranteed stepsize lower bound becomes $\lambda_k\geq \rho/(3L)$. Hence, larger values of $\rho$ provide a larger guaranteed stepsize lower bound at the expense of a weaker descent guarantee, whereas smaller values of $\rho$ yield stronger descent at the expense of a smaller stepsize lower bound. This result can also be applied to the accelerated case (Section~\ref{Accelerated adaptive stepsize} and Theorem~\ref{theorem3}), which requires modifying the momentum recursion in \ref{alg:Algorithm2}.
\end{remark}
}
\subsection{Non-accelerated smooth minimization}\label{Non-accelerated smooth minimization}
In this subsection, we refine the result proposed in the previous subsection for (locally)~smooth minimization, i.e., $h(x) = 0$ in~\eqref{composite problem}. Specifically, the analysis in Theorem~\ref{theorem1} can also be applied to a (locally)~smooth minimization. Note that in the smooth case, $G^{f}_{\lambda h}(x) = \nabla f(x)$ and the linesearch in \ref{alg:Algorithm1} is simplified to finding the largest $\lambda$ satisfying $\phi(2\lambda) \leq \phi(\lambda) + \dfrac{\lambda}{2}\phi^\prime(0)$. The following corollary explains the results when we invoke \ref{alg:Algorithm1} to minimize the (locally)~smooth function.
\begin{corollary}[Convergence of smooth minimization] \label{corollary-smooth-nonaccelerated}
Let the function~$F$ in \eqref{composite problem} be smooth (i.e., $h(x) = 0$). Then, the stepsize rule in \ref{alg:Algorithm1} reduces to \eqref{partial stepsize rule} while ensuring Theorem \ref{theorem1} results of the uniform stepsize lower bound $\lambda_k \geq {1}/{3L}$ and the sub-optimality bound $F(x_k) - F(x^*) \leq \dfrac{D}{k}$.
\end{corollary}
{\paragraph{{\bf Comparison of different stepsize regimes:}} 
We close this subsection by highlighting a common feature shared between the Polyak stepsize~\eqref{Polyak stepsize}, Armijo stepsize rule \eqref{Armijo-rule}, and our proposed stepsize rule~\eqref{partial stepsize rule}. These three stepsizes, denoted by $\lambda$, respect the following inequalities 
\begin{subequations}\label{G_interval}
\begin{align}\label{inerval_1}
    \underline{\lambda} \leq \lambda \leq c\,\dfrac{f(y_1) - f(y_2)}{\|\nabla f(x_k)\|^2},
\end{align}
where the points~$y_1$ and $y_2$ potentially depend on $x_k$ and the stepsize $\lambda$. More specifically, we have 
\begin{align}
\label{interval_2}
\begin{cases}
\begin{array}{llllllll}
\text{Polyak}\,\eqref{Polyak stepsize}:&y_1 = x_k,  &y_2 &= x^*, &c &= 1, &\underline{\lambda} &= \frac{1}{2L},\\
\text{Armijo}\,\eqref{Armijo-rule}:&y_1 = x_k,  &y_2 &= x_k - \lambda \nabla f(x_k), &c &> 1, &\underline{\lambda} &= \frac{2(1-{1}/{c})}{L},\\
\text{Our stepsize}\,\eqref{partial stepsize rule}:&y_1 = x_k - \lambda \nabla f(x_k),  &y_2 &= x_k - 2\lambda \nabla f(x_k), &c &= 2, &\underline{\lambda} &= \frac{1}{3L}. 
\end{array}
\end{cases}
\end{align}
\end{subequations}}
{\textcolor{black}{Inequalities~\eqref{inerval_1}--\eqref{interval_2} show the well-known lower bounds on the stepsize in terms of the smoothness parameter $L$ for Polyak- and Armijo-type rules.}
Our proposed rule in Theorem~\ref{theorem1} instead achieves a lower bound of $1/(3L)$, which is smaller than the bound of the existing linesearch rules. While larger stepsizes are commonly expected to yield faster convergence, our numerical experiments across multiple problem classes show that the proposed linesearch consistently outperforms existing methods. A possible explanation for this is that the theoretical bounds are typically conservative: in realistic problems, the local geometry often permits stepsizes much larger than those implied by the global smooth constant. Our linesearch procedure exploits this favorable curvature by accepting substantially larger steps in most iterations, leading to faster objective decrease and fewer gradient/proximal evaluations overall.}

Notice that the optimal choice of the stepsizes is the upper bound of the admissible interval \eqref{inerval_1}. It is also important to note that in the case of Armijo \eqref{Armijo-rule} and our proposed choice \eqref{partial stepsize rule}, this upper bound also depends on $\lambda$ (see \eqref{interval_2}), making the stepsize rule implicit. We also wish to draw attention to the choice of the upper bound between Armijo and our proposed approach in \eqref{inerval_1}: One can inspect that the proposed rule \eqref{partial stepsize rule} looks one step ahead of Armijo in the direction of $\nabla f$. This may explain why the proposed rule \eqref{partial stepsize rule} performs numerically better, as it can approximate the curvature of the function~$f$ by ``looking one step further" than Armijo \textcolor{black}{(note that this look-ahead factor cannot be chosen arbitrarily large while preserving the required decrease property; see Remark~\ref{Generalization of the linesearch rule})}. It is also worth noting that the recent work~\cite{Malitsky2020} offers an adaptive stepsize rule with the lower bound $1/2L$ and an explicit upper bound (i.e., independent of~$\lambda$) that depends on the gradient of two successive iterations (see \eqref{malitsky ad-stepsize}).

\textcolor{black}{Finally, we would like to emphasize the differences between the proposed stepsize rule and Wolfe-type linesearch methods~\eqref{Wolfe-rule}--\eqref{strong-Wolfe-rule}, which further clarify the motivation for our approach. Although Wolfe-type rules can also provide a lower bound on the stepsize under appropriate conditions~\cite[Section~3.1]{Nocedal2006}, the following distinctions highlight the practical and structural differences between these approaches.\\
\noindent\textbf{(a) Lower oracle cost during the linesearch.}
Standard Wolfe-type conditions require gradient evaluations at trial points to verify the curvature condition. In contrast, our linesearch rule requires only function-value evaluations at trial points, while $\nabla f(x_k)$ is evaluated once per outer iteration. This can be advantageous when gradient evaluations are substantially more expensive than function evaluations.\\
\noindent\textbf{(b) Compatibility with composite optimization.}
Wolfe-type conditions are formulated in terms of gradient information and are naturally suited to smooth optimization. In contrast, our rule is formulated using function values at trial points and the gradient at the current iterate, making it naturally compatible with proximal-gradient methods and composite problems with a nonsmooth component. A formal treatment of stochastic or inexact oracle settings, however, is beyond the scope of this work.\\
\noindent\textbf{(c) Geometric interpretation.}
As illustrated in Figure~\ref{fig:linesearch methods}, our rule uses information from the one-dimensional function $\phi(\cdot)$ at two stepsizes to characterize the acceptable stepsize region. In particular, the upper endpoint of the resulting acceptable region is typically closer to the minimizer $\lambda^*$ of $\phi(\cdot)$ than the corresponding upper endpoints obtained from the Armijo ($\lambda^A$) and Wolfe ($\lambda^W$) conditions. This provides an intuitive explanation for why the proposed rule can identify stepsizes that are better aligned with the local geometry of the objective function. In contrast, Wolfe-type conditions characterize acceptable stepsizes through separate sufficient-decrease and curvature conditions. Thus, although both approaches may produce an interval of acceptable stepsizes, the criteria defining these intervals and the information used to construct them are fundamentally different.}

{\color{black}
We conclude this subsection with the following remark, which discusses a generalization of the stepsize rule~\eqref{partial stepsize rule} in the smooth setting.
\begin{remark}[Generalization of the linesearch rule in the smooth case]\label{smooth-Generalization of the linesearch rule}
The following generalization of the stepsize rule~\eqref{partial stepsize rule} can be applied to the smooth case~($h = 0$):
        \begin{equation} \label{eq:smooth-gen}
          \arg\max\limits_{\lambda} \left \{\lambda \,|\,\,  \phi_k(\alpha\lambda) \leq \dfrac{\rho}{2}\phi_k^{\prime}(0)\lambda + \phi_k(\lambda) \right\}, \, (\rho,c_1) \in (0,2), \, \alpha \leq \frac{\rho}{c_1}+1,\, \alpha \geq \frac{\rho+2}{2},
        \end{equation}
 Under this stepsize rule, the following hold:
 \begin{itemize}
     \item $F(x_{k+1})-F(x_k) \leq \frac{(c_1-2)\lambda_k}{2} \|\nabla f(x_k)\|^2,$
     \item The stepsize is lower bounded by $\frac{-\rho + 2(\alpha - 1)}{(\alpha^2 - 1)L}$.
 \end{itemize}
\end{remark}
\begin{proof}
We begin with Lemma~\ref{lema2}, which, under the smoothness assumption, can be equivalently written as
\begin{align}\label{IncBound-smooth}
    F(x^+) - F(z) \leq \langle x^+ - x,  \nabla f(x^+) - \frac{c_1}{2}\nabla f(x) \rangle - \Big( \frac{2-c_1}{2\lambda} \|x^+ - x\|^2 + \frac{1}{\lambda}\langle x^+ -x, x - z \rangle \Big).
\end{align}
Following the same argument as in the proof of Lemma~\ref{lema4linesearch}, we observe that employing the stepsize rule~\eqref{eq:smooth-gen} guarantees that $\langle x^+ - x,  \nabla f(x^+) - \frac{c_1}{2}\nabla f(x) \rangle \leq 0$ in \eqref{IncBound-smooth}. Setting $z = x$ in \eqref{IncBound-smooth} yields the first result (by knowing that $x_{k+1} = x^+$ and $x_k = x$).

To derive the lower bound for the stepsize, we follow the same procedure as in the proof of Theorem~\ref{theorem1}. Using the stepsize rule~\eqref{eq:smooth-gen} and the convexity of $\phi$, inequality~\eqref{thr2_3_ineq1} can be equivalently rewritten as
\begin{align}\label{thr2_3_ineq1_revised}
    0 \leq \left\langle \nabla f(x_k), \nabla f\bigl(x_k - \lambda \nabla f(x_k)\bigr) - \left(\frac{\rho+L(\alpha-1)^2 \lambda}{2(\alpha - 1)}\right)\nabla f(x_k) \right\rangle.
\end{align}

Comparing this expression with inequality~\eqref{eq4} for the smooth case, which holds universally, we deduce the lower bound from the following relation, which concludes the proof:
\begin{align*}
    \lambda L - 1 = - \left(\frac{\rho+L(\alpha-1)^2 \lambda}{2(\alpha - 1)}\right) \Rightarrow \lambda = \frac{-\rho + 2(\alpha - 1)}{(\alpha^2 - 1)L}.
\end{align*}
Note that our original stepsize rule for the smooth case, given in \eqref{partial stepsize rule}, is recovered by setting $\alpha = 2$, $\rho = 1$, and $c_1 = 1$.
\end{proof}

}
\section{Accelerated Adaptive Stepsize}\label{Accelerated adaptive stepsize}
This section pertains to the analysis of convergence for the update direction $d_k$ in accelerated methods. The general form of our accelerated minimization algorithm is provided in \ref{alg:Algorithm2}, which follows the accelerated proximal gradient descent method with the difference in the choice of stepsize $\lambda$ \cite{Nesterov1983, Nesterov2018, bubeck2015convex}. Differently from \cite{Nesterov1983,Beck2009}, the stepsize is not fixed, and in our numerical experience, it does not strictly decrease in many iterations, which partially explains the acceleration of the algorithm up to the optimal theoretical bound $\mathcal{O}(k^{-2})$.
\begin{theorem}[Accelerated adaptive stepsize]\label{theorem3}
Consider the function~$F$ in \eqref{composite problem} as a composition of a smooth function $f(x)$ and a prox-friendly function $h(x)$. The sequence $\left (x_k \right)_{k \in \mathbb{N}}$ generated by \eqref{stepsize-search direction} with the direction and stepsize rule 
 \begin{equation*}\label{alg:Algorithm2}
\tag{Alg.\,2}
        \begin{cases}
            \beta_k = \dfrac{1+\sqrt{1+4\beta_{k-1}^2}}{2}, \,\, \gamma_k = \dfrac{1-\beta_k}{\beta_{k+1}}, \,\, x_0 = x_1, \,\, \beta_0 = 0,\\
            \lambda_k = \arg\max\limits_{\lambda \in \mathbb{R}} \left\{\lambda \,|\,\, \phi_k(2\lambda) \leq \phi_k(\lambda) - \lambda \langle G^{f}_{\lambda h}(x_k), \nabla f(x_k) \rangle + \dfrac{\lambda}{2}\|G^{f}_{\lambda h}(x_k)\|^2, \, \lambda \leq \lambda_{k-1} \right\},\\
            d_k = -G^{f}_{\lambda_k h}(x_k)-\gamma_k \left(\dfrac{\lambda_{k-1}}{\lambda_k} d_{k-1} - G^{f}_{\lambda_k h}(x_k) + \dfrac{\lambda_{k-1}}{\lambda_k} G^{f}_{\lambda_{k-1} h}(x_{k-1}) \right). 
        \end{cases}
\end{equation*}
ensures the uniform stepsize lower bound $\lambda_k \geq {1}/{3L}$ and the sub-optimality bound $F(x_k) - F(x^*) \leq \dfrac{D'}{k^2}$ where $L$ is the smoothness constant of $f$, and $D'$ is a constant that only depends on the initial condition $x_0$, the optimal solution $x^*$, and the constant~$L$.
\end{theorem}
\begin{proof}
The proof of the first proposition is the same as in Theorem \ref{theorem1}. Specifically, according to stepsize rule in \ref{alg:Algorithm2}, we should first check whether the previous $\lambda_{k-1}$ satisfies the linesearch condition. Since the inequality \eqref{thr2_3_ineq2} holds for $\lambda = {1}/{3L}$ (according to the linesearch condition in \ref{alg:Algorithm2}), we have the guarantee that $\lambda_k \geq {1}/{3L}$. 
\textcolor{black}{For the convergence analysis, we first note that~\ref{alg:Algorithm2} can be equivalently written in the standard accelerated proximal-gradient form~\cite[Section~10.7]{Beck2017} as follows:
\begin{align}\label{eq: Nes-like iter}
y_{k+1} &= x_k - \lambda_k G^{f}_{\lambda_k h}(x_k),\, y_1 = x_1\nonumber\\
x_{k+1} &= (1-\gamma_k)y_{k+1}+\gamma_k y_k.
\end{align}
Thus, $y_{k+1}$ represents the proximal-gradient iterate generated from $x_k$. We use this equivalent formulation to make the role of the sequence ${y_k}$ explicit. Regarding the nonasymptotic convergence bound,} in the first step, we invoke the increment inequality of \ref{eq3} in Lemma~\ref{keylemma} for two cases of $z = y_{k}$ and $z = x^*$, which arrives at
\begin{subequations}
\begin{align}
    & F(y_{k+1}) - F(y_{k}) \leq -\dfrac{1}{2\lambda_k} \|y_{k+1} - x_k\|^2 - \dfrac{1}{\lambda_k}\langle y_{k+1} -x_k, x_k - y_{k} \rangle, \label{eq9} \\
    & F(y_{k+1}) - F(x^*) \leq -\dfrac{1}{2\lambda_k} \|y_{k+1} - x_k\|^2 - \dfrac{1}{\lambda_k}\langle y_{k+1} -x_k, x_k - x^*\rangle. \label{eq10}
\end{align}
\end{subequations}
Let us define $\delta_k := F(y_{k}) - F(x^*)$. Then, multiplying \eqref{eq9} by $\beta_k - 1$, and adding the two sides of the inequality to \eqref{eq10} yields
\begin{equation*}
    \beta_k \delta_{k+1} - (\beta_k - 1)\delta_k \leq - \dfrac{\beta_k}{2\lambda_k}\|y_{k+1} - x_k\|^2 - \dfrac{1}{\lambda_k}\langle y_{k+1} - x_k, \beta_k x_k - (\beta_k - 1)y_{k} - x^* \rangle.
\end{equation*}
Multiplying the above inequality by $\lambda_k \beta_k$, \textcolor{black}{considering the update rule for the coefficient $\beta_k$ in~\ref{alg:Algorithm2}, given equivalently by $\beta_{k-1}^2 = \beta_k^2 - \beta_k$,} and using the fact that $\lambda_k \leq \lambda_{k-1}$, we obtain
\begin{equation}
\begin{split}\label{eq11}
\lambda_k \beta_{k}^2 \delta_{k+1} - \lambda_{k-1} \beta_{k-1}^2 \delta_{k} \leq - \dfrac{1}{2} \Big(\|\beta_k (y_{k+1} - x_k)\|^2 + 2\beta_k \langle y_{k+1} - x_k, \beta_k x_k - (\beta_k - 1)y_{k} - x^* \rangle \Big).   
\end{split}
\end{equation}
The right-hand side of \eqref{eq11} can be equivalently written as
\begin{align}\label{eq12}
& \blank{-1.3cm} \|\beta_k (y_{k+1} - x_k)\|^2 + 2\beta_k \langle y_{k+1} - x_k, \beta_k x_k - (\beta_k - 1)y_{k} - x^* \rangle = \nonumber \\
& \blank{4cm} \|\beta_k y_{k+1} - (\beta_k - 1)y_{k} - x^*\|^2 - \|\beta_k x_k - (\beta_k - 1)y_{k} - x^*\|^2.   
\end{align}
Substituting \eqref{eq12} into \eqref{eq11} and rearranging the inequality lead to
\begin{equation}\label{eq13}
\lambda_k \beta_{k}^2 \delta_{k+1} - \lambda_{k-1} \beta_{k-1}^2 \delta_{k} \leq - \dfrac{1}{2} \Big(\|\beta_k y_{k+1} - (\beta_k - 1)y_{k} - x^*\|^2 - \|\beta_k x_k - (\beta_k - 1)y_{k} - x^*\|^2 \Big).
\end{equation}
Using the update rule of $x_{k+1}$ on the right-hand side of \eqref{eq13} reduces to
\begin{equation}\label{eq14}
\beta_k y_{k+1} - (\beta_k - 1)y_{k} - x^*= \beta_{k+1} x_{k+1} - (\beta_{k+1} - 1)y_{k+1} - x^*,  
\end{equation}
 which is equivalent to
\begin{equation*}
x_{k+1}=\dfrac{(-1 + \beta_k + \beta_{k+1})}{\beta_{k+1}}y_{k+1} + \dfrac{1 - \beta_k}{\beta_{k+1}}y_{k}.  
\end{equation*}
By combining (\ref{eq13}) and (\ref{eq14}) and defining $u_k = \beta_k x_k - (\beta_k - 1) y_{k} - x^*$, we obtain
\begin{equation}\label{eq15}
\lambda_k \beta_{k}^2 \delta_{k+1} - \lambda_{k-1} \beta_{k-1}^2 \delta_k \leq \dfrac{1}{2}(\|u_k\|^2 - \|u_{k+1}\|^2).
\end{equation}
We note that the inequality~\eqref{eq15} can be viewed as the so-called Lyapunov function as it indicates that the term $\lambda_k \beta_{k}^2 \delta_{k+1} + \|u_{k+1}\|^2/2$ is monotonically decreasing in each iteration. Summing inequalities (\ref{eq15}) from $ k = 1$ to $k = T$, one obtains 
\begin{equation*}
\lambda_{T}\beta_{T}^2 \delta_{T+1} - \lambda_0\beta_{0}^2 \delta_{1} \leq \dfrac{1}{2}\|u_1\|^2 - \dfrac{1}{2}\|u_{T+1}\|^2 \leq \dfrac{1}{2}||u_1||^2.
\end{equation*}
Setting $\beta_0=0$ in the above equation yields $\delta_{T+1} \leq \dfrac{\|u_1\|^2}{2 \lambda_{T}\beta_{T}^2}$.
Finally, using Remark \ref{stepsize appr} and the growing rate of $\beta_k \geq k/2$, we deduce the desired result
\begin{equation*}
 F(y_{T+1}) - F(x^*) = \delta_{T+1} \leq \dfrac{6L\|u_1\|^2}{C T^2} \leq \dfrac{D'}{T^2}. 
\end{equation*}
\end{proof}
We can refine the result proposed in Theorem \ref{theorem3} for smooth minimization, where the non-smooth part, $h$, in \eqref{composite problem} is $h(x) = 0$. In this case, $G^{f}_{\lambda h}(x) = \nabla f(x)$, and the linesearch in the stepsize rule \ref{alg:Algorithm2} is simplified to find the largest positive scalar $\lambda$ that satisfies $\phi(2\lambda) \leq \phi(\lambda) + \frac{\lambda}{2}\phi^\prime(0)$. This is formalized in the next Corollary.
\begin{corollary}[Accelerated smooth minimization] \label{remark for smooth}
Let the function~$F$ in \eqref{composite problem} be smooth (i.e., $h(x) = 0$). Then, the stepsize rule in \ref{alg:Algorithm2} reduces to \eqref{partial stepsize rule} while ensuring Theorem \ref{theorem3} results of the uniform stepsize lower bound $\lambda_k \geq {1}/{3L}$ and the sub-optimality bound $F(x_k) - F(x^*) \leq \dfrac{D'}{k^2}$.
\end{corollary}
{\color{black} We close this section by discussing the non-increasing stepsize restriction in~\ref{alg:Algorithm2} and a potential direction for relaxing this constraint.
\begin{remark}[Stepsize monotonicity] \label{remark:nonincreasing stepsize}
The non-increasing stepsize condition $\lambda_k \leq \lambda_{k-1}$ is required in Theorem~\ref{theorem3} to guarantee the decrease of the gap function $\delta_k$ in~\eqref{eq11}. This restriction is closely related to the acceleration momentum, which can induce oscillatory trajectories across regions with heterogeneous local smoothness~\cite{su2016differential}. Empirically, $\lambda_k$ predominantly decreases during sharp trajectory transitions and remains unchanged during smoother phases (see Figure~\ref{fig:maxcut_sdp}). Relaxing this restriction for general convex accelerated methods remains an open challenge~\cite[Section~3]{Malitsky2020}. A promising direction is to incorporate restarting strategies that periodically reset the momentum and stepsize, potentially allowing the stepsize to increase after a restart~\cite{Emmanuel2015}.
\end{remark}}
\section{Numerical Results}\label{Numerical results}
We illustrate the performance of our proposed stepsize rules~\ref{alg:Algorithm1} and \ref{alg:Algorithm2} in three major problem classes and {real-world applications datasets} widely used in operations research, machine learning, and systems and control:  

\begin{enumerate}[leftmargin=8mm, itemsep=1pt, topsep = 0mm]
    \item \textbf{Smooth functions}, including the specific objective functions 
    \begin{enumerate}[label=(\roman*), leftmargin=8mm, itemsep=1pt, topsep = 0mm]
        \item logistic regression, motivated by classification tasks in machine learning problems~\cite{jiang2020online}, 
        
        \item quadratic programming, motivated by portfolio optimization and resource allocation problems~\cite{huang2025restarted, mittal2020robust, mehrotra1990algorithm}),  
        
        \item log-sum-exp, motivated by discrete-choice modeling and entropy-regularized transport, e.g., multinomial logit model and softmax demand models~\cite{feng65technical, qu2025sinkhorn,liu2020assortment}, and 
        
        \item max-cut, a benchmark for combinatorial optimization and approximate semidefinite programming~\cite{boyd1997semidefinite, nesterov2007smoothing, han2025low}. 
    \end{enumerate}
    
    \item \textbf{Composite minimization}, including the specific objective functions 
    
    \begin{enumerate}[label=(\roman*), leftmargin=8mm, itemsep=1pt, topsep = 0mm]
    
    \item $\ell_1$-regularized least squares, commonly used in high-dimensional statistics for support recovery and in signal processing for sparse signal reconstruction~\cite{neykov2016l1, selesnick2017sparse}), 
    
    \item $\ell_1$-constrained least squares,  in sparse signal recovery with explicit coefficient bounds \cite{gaines2018algorithms}), and
    
    \item~$\ell_1$-regularized logistic regression used in sparse classification and feature selection, \cite{bertsimas2021sparse}.
    \end{enumerate}

    \item \textbf{Non-convex minimization}, where we consider a cubic objective function relevant to cubic-regularization in nonconvex optimization and trust-region subproblems~\cite{Nesterov2006, chen2022accelerating}.

    \item {\textbf{Real-world applications and comparison with Armijo rule}, where various simulations are executed across multiple algorithms, including the Armijo linesearch, using real-world datasets of different dimensionality from LIBSVM~\cite{chang2011libsvm}.}
\end{enumerate}

Before presenting the numerical results, let us first start with a practical discussion concerning the implementation of the approximate solution in Remark~\ref{stepsize appr} for the stepsize rule in \ref{alg:Algorithm1}. 
The backtracking in Remark~\ref{stepsize appr} uses an initial constant $\lambda_0$ to approximate the proposed stepsize rule in \ref{alg:Algorithm1}. On the one hand, if the initial $\lambda_0$ is too large, it takes time to compute the desired stepsize satisfying the linesearch condition. On the other hand, a too small $\lambda_0$ deteriorates the convergence speed \cite{Nocedal2006}. To address this issue, we approximate the function $\phi$ in \eqref{phi func} via a quadratic function, denoted by $\widetilde{\phi}$, and suggest the desired $\lambda_0$ as its optimizer; see Figure \ref{fig:img2} for a visual representation of this approximation (cf. Figure \ref{fig:linesearch methods}).  
\begin{figure}
     \centering
           \includegraphics[width=0.6\textwidth]{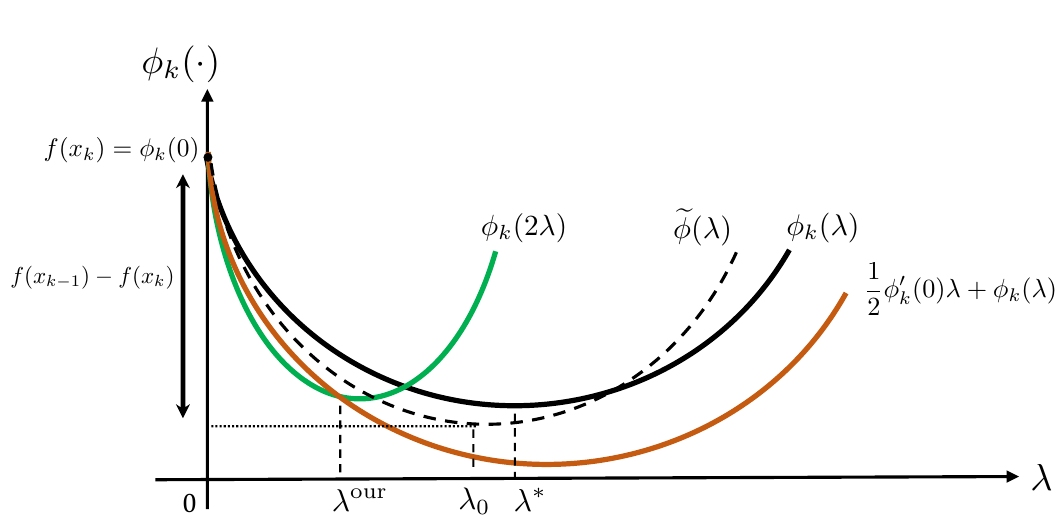}
           \caption{Initial choice of $\lambda_0$ at the $({k+1})^{\text{th}}$ iteration.}
           \label{fig:img2}
\end{figure}
For the approximate function $\widetilde{\phi}$, we suggest the following criteria:
\begin{equation*}
    \begin{rcases}
      & \widetilde{\phi}(0) = \phi(0),\\
      & \widetilde{\phi}^\prime(0) = \phi^\prime(0), \\
      & \widetilde{\phi}(0) - \min \widetilde{\phi}(\lambda) = f(x_{k-1}) - f(x_k)
    \end{rcases}
    \implies \lambda_0 = \arg\min_{\lambda} \widetilde{\phi}(\lambda) = \dfrac{2\Big(f(x_{k-1}) - f(x_k)\Big)}{\|\nabla f(x_k)\|^2}.
\end{equation*}
Notice that the function~$\widetilde{\phi}$ is constructed in a way whose minimum value is $f(x_{k-1}) - f(x_k)$, i.e., the objective improvement of the last iteration. {Note that Theorem~\ref{theorem1} establishes a lower bound for the stepsize. However, the stepsize rule is satisfied for any $\lambda \leq \frac{1}{3L}$; thus, if $\lambda_0 \leq \frac{1}{3L}$, the condition holds consistently and finite termination is guaranteed. This property follows directly from \eqref{thr2_3_ineq2}, which remains valid for all $\lambda \in (0, \frac{1}{3L}]$. The backtracking diminishing factor is set to $0.8$ for both the proposed rule~\eqref{back-line} and the Armijo line search rule~\eqref{Armijo-rule} across all simulations.}
To evaluate our performance, we compare our proposed algorithms with the following methods from the literature: For the class of (1)~smooth optimization, we consider the gradient descent (GD)~\cite{Nocedal2006}, Nesterov's accelerated gradient descent (NAGD)~\cite{Nesterov1983}, the recent adaptive gradient descent (AdGD) and its heuristic accelerated version~(AdGD-accel)~\cite{Malitsky2020}. Respectively, for the class of (2)~composite minimization, we consider the nonsmooth counterparts including the proximal gradient descent (ISTA) and its accelerated version~(FISTA)~\cite{Beck2009}, the adaptive golden ratio algorithm~(aGRAAL)~\cite{malitsky2020golden} and the adaptive primal-dual algorithm (APDA)~\cite{vladarean2021first}. For (3)~non-convex minimization, we use the existing algorithms from the first class of (1)~smooth minimization. {We also note that, for a more objective comparison, all methods utilize the same number of gradient oracle calls, a primary determinant of computational complexity in first-order algorithms.}

\subsection{Smooth minimization}\label{subsec: smooth-numerics}
This section includes four different classes of smooth functions: 

\paragraph{{\bf (i) Logistic regression: }} Motivated by classification problems \cite{jiang2020online}, we consider logistic regression with quadratic regularization as follows:
\begin{equation}
\label{log-reg}
    \min_{x\in\mathbb{R}^d} \dfrac{1}{N} \sum_{i=1}^{N} \log \left(1+\exp(-b_i a_{i}^\top x)\right) + \dfrac{\gamma}{2} \|x\|^2,
\end{equation}
where $N$ is the number of data points, $A = [a_1|a_2|\cdots|a_n]\in \mathbb{R}^{d\times N}$ and $\{b_i\}_{i=1}^N \in \mathbb{R}$ are the features and labels of data, and $\gamma$ is a regularization parameter proportional to ${1}/{N}$. In our experiments, the test data is generated as follows: the matrix $A \in \mathbb{R}^{N\times N}$ is generated randomly using the standard Gaussian distribution $\mathcal{N}(0, 1)$, where we generate $A$ with $50\%$ correlated columns as $A(:,j+1) = 0.5A(:,j) + \textrm{randn}(:)$. The observed measurement vector $b$ is generated as $b := Ax^{\natural} + \mathcal{N}(0, 0.05)$, where $x^{\natural}$ is generated randomly using $\mathcal{N}(0, 1)$, $d = N = 200$, and we estimate the smooth constant of the gradient by the closed-form expression $L = \|A\|^2 / N + \gamma$ used in GD and NAGD as a stepsize. Figure \ref{fig:img4} reports the results where all algorithms are initialized at the same point chosen randomly.

\paragraph{{\bf (ii) Quadratic programming:}} Following recent works on portfolio and robust optimization~\cite{huang2025restarted, mittal2020robust}, we examine the quadratic objective function
\begin{equation}
\label{quad}
    \min_{x\in\mathbb{R}^d} \dfrac{1}{2}x^\top B x + b^\top x,
\end{equation}
where $B\in \mathbb{R}^{d\times d} > 0$ and $b\in \mathbb{R}^d$ are given. This function is smooth with constant $L = \|B\|$. The parameters $B$ and $b$ are chosen in the same manner as in the previous part, with the difference that $B = A^\top A$ and $b$ are normalized. Additionally, in our experiments, we set $d = 200$. The results of minimizing this quadratic function are provided in Figure~\ref{fig:img5} where all the methods are initialized at the same point. It is interesting to note that the proposed stepsize rule for accelerated algorithms exhibits similar behavior to NAGD with comparative performance while outperforming other adaptive algorithms.

\paragraph{{\bf (iii) Log-Sum-Exp:}} In the context of discrete-choice modeling and entropy-regularized transport \cite{feng65technical, qu2025sinkhorn,liu2020assortment}, we consider the log-sum-exp objective function with regularization
\begin{equation}
\label{log-sum-exp}
 \min_{x\in\mathbb{R}^d} \log \left(\sum_{i=1}^{N} \exp(a_{i}^\top x - b_i)\right) + \dfrac{\gamma}{2} \|x\|^2,  
\end{equation}
where $A = [a_1|a_2|\cdots|a_n]\in \mathbb{R}^{d\times N}$ and $\{b_i\}_{i=1}^N \in \mathbb{R}$ are the problem data. The log-sum-exp function is viewed as a smooth approximation of $\max_{x\in\mathbb{R}^{d}} \{a_1 ^\top x - b_1, a_2 ^\top x - b_2, \cdots, a_N ^\top x - b_N\}$. This function can sometimes be ill-conditioned, making the minimization task particularly difficult for the first-order methods. The function is smooth with the smoothness constant $L = \max_i \|a_i\|^2 + \gamma$.
In the experiment, we set $N = d = 200$, and $A$ and $\{b_i\}_{i=1}^N$ are generated using a similar approach as in the logistic regression part. The results of minimizing the log-sum-exp function are depicted in Figure~\ref{fig:img6} where the proposed accelerated method demonstrates superior efficacy compared to other algorithms.
\begin{figure}[t]
    \centering
    \captionsetup{justification=centering}
    \includegraphics[scale=0.30]{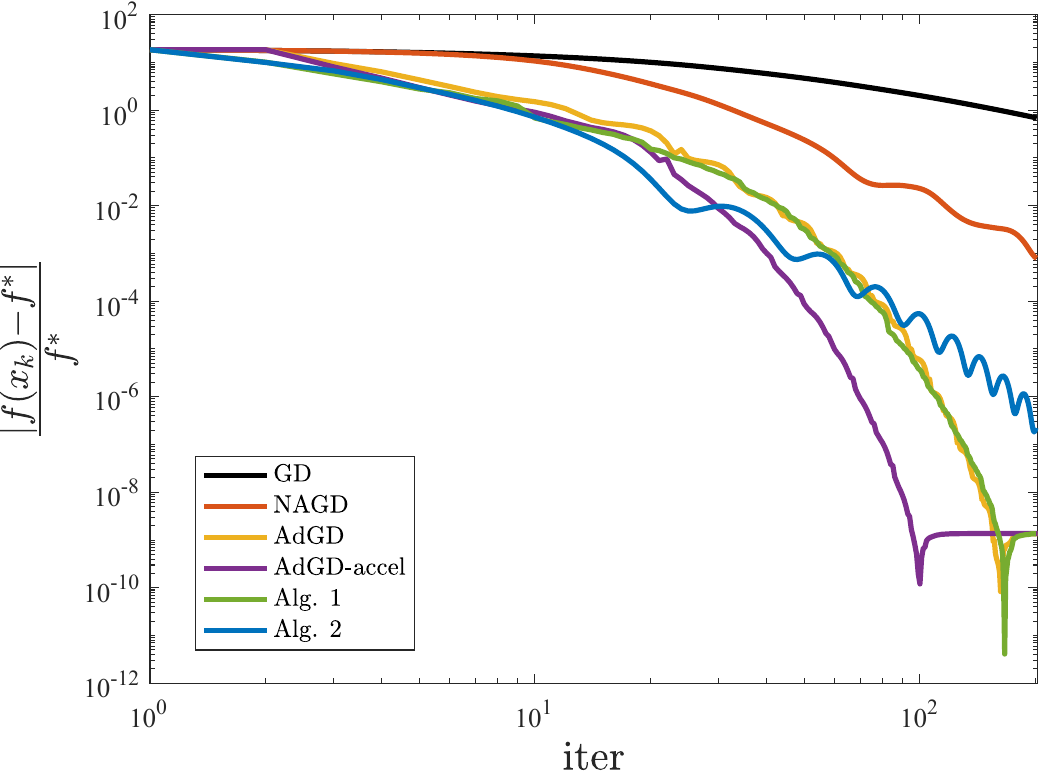}
    \includegraphics[scale=0.30]{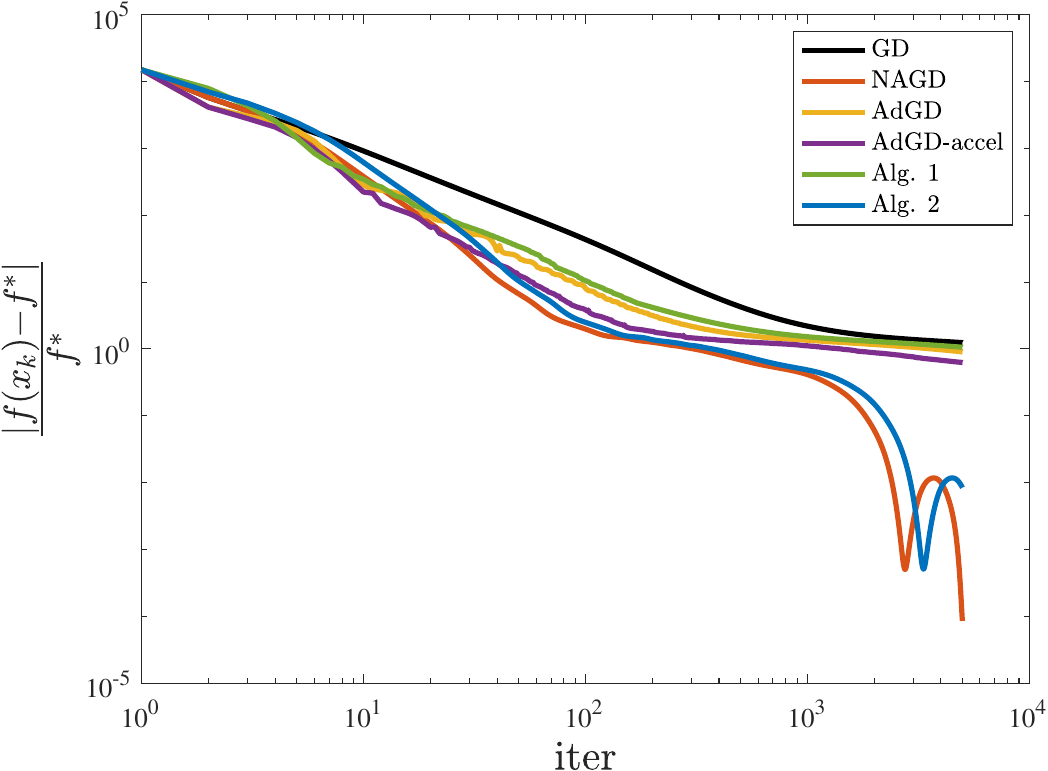}
    \includegraphics[scale=0.30]{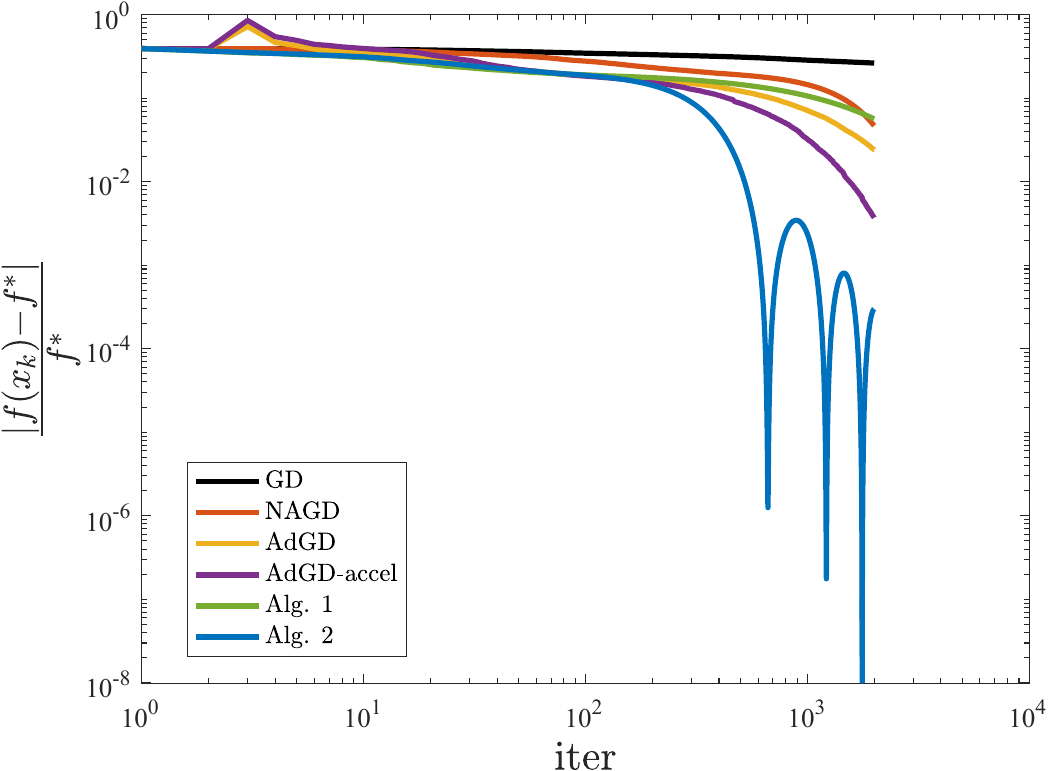}\\
    \subfloat[Logistic regression~\eqref{log-reg}]{\label{fig:img4}\includegraphics[scale=0.312]{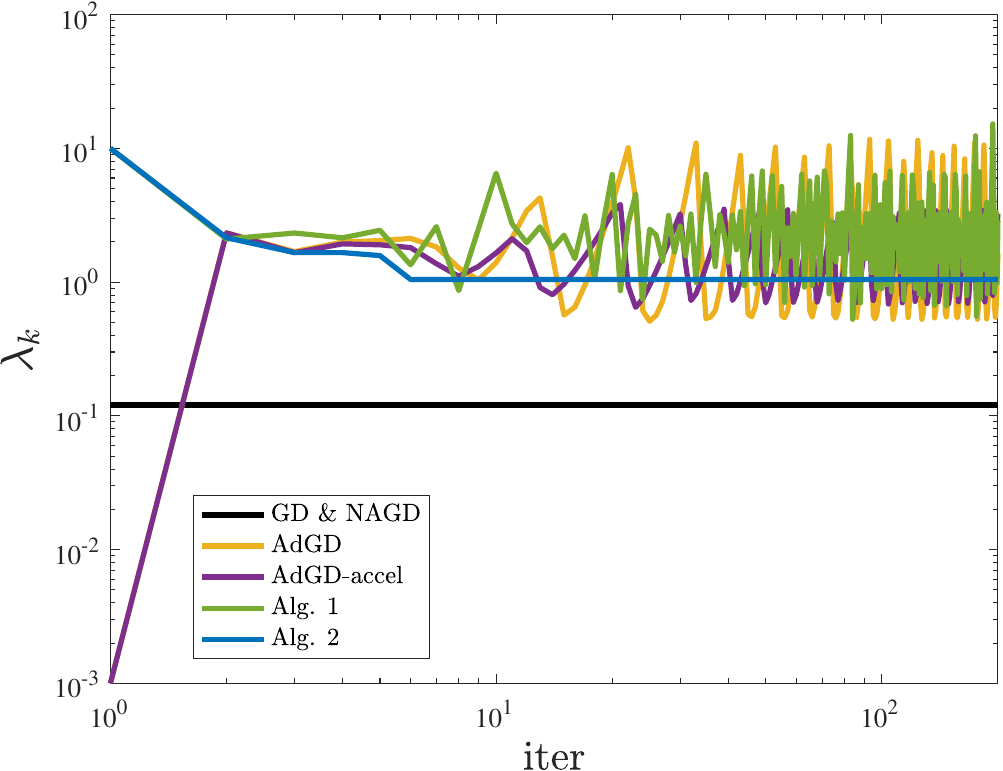}}
    \subfloat[Quadratic programming~\eqref{quad}]{\label{fig:img5}\includegraphics[scale=0.312]{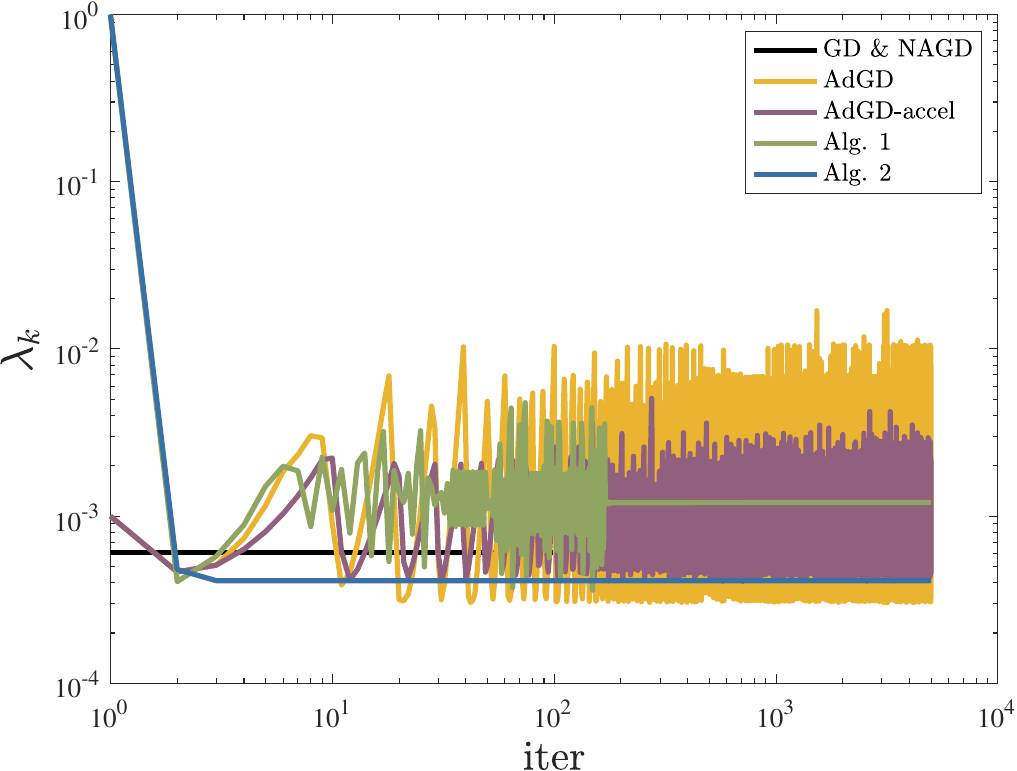}}
    \subfloat[Log-Sum-Exp~\eqref{log-sum-exp}]{\label{fig:img6}\includegraphics[scale=0.312]{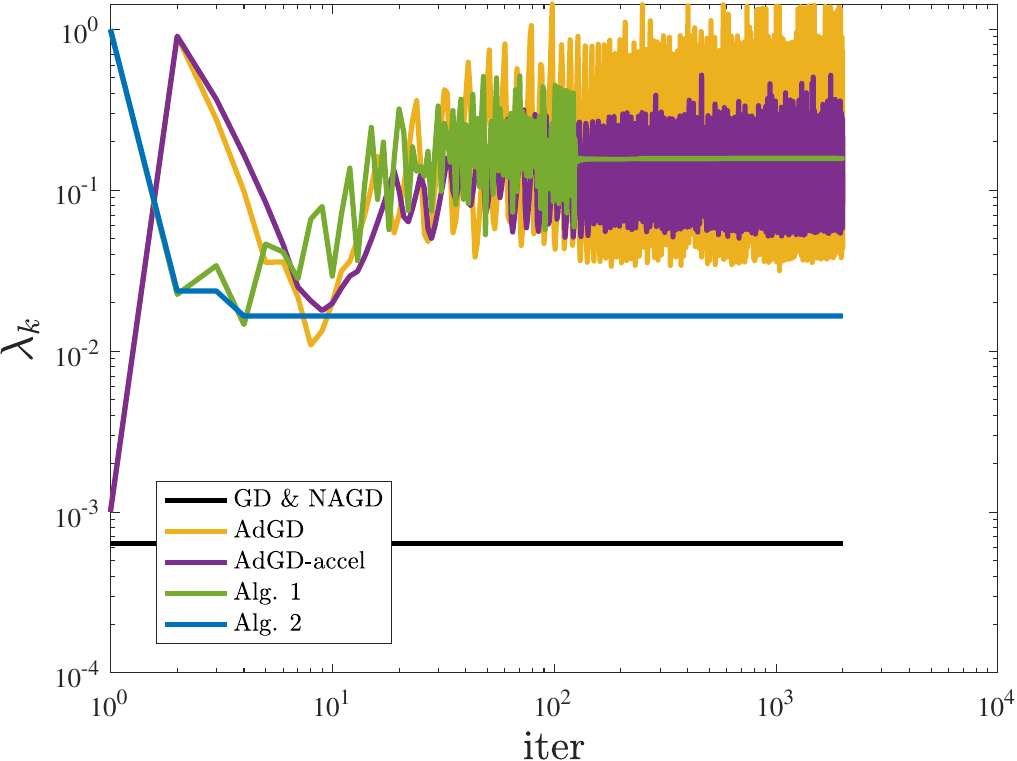}}
    \caption{The results for the class (1)~smooth minimization. The first row shows the optimality gap, and the second row shows the stepsize behavior.}
    \label{fig:imgsmooth}
\end{figure}

\paragraph{\bf (iv) Approximate semidefinite programming:} Semidefinite programs are ubiquitous in combinatorial optimization and control~\cite{boyd1997semidefinite, El-Ghaoui, han2025low}, graph theory~\cite{tang2024solving}, network and communication~\cite{biswas2006semidefinite}. Most of the problems in this class essentially deal with minimizing the maximum eigenvalue of a matrix, which is a nonsmooth function of its variable. A popular example falling into this category is the \textit{max-cut} problem, whose $\varepsilon$-approximation with the regularizer~$\mathcal{R}(\cdot)$ reads as 
\begin{equation}\label{regularized maxcut dual}
\min\limits_{y \in \mathbb{R}^n}  f_{\varepsilon}(C+\text{diag}(y)) - \langle \mathbf{1}, y \rangle + \eta \mathcal{R}(y), \quad 
f_{\varepsilon}(X) = \varepsilon \log \left(\sum_{i = 1}^n \exp(\lambda_i(X)/\varepsilon)\right),
\end{equation}
where $C$ is a constant matrix, $\eta$ is a regularization coefficient, and $f_{\varepsilon}(X)$ is the smooth convex approximation of $\lambda_{\textbf{max}}(X)$. As shown in \cite{nesterov2007smoothing}, the gradient of $f_{\varepsilon}(X)$ can be computed as 
\begin{equation*}
    \nabla f_{\varepsilon}(X) =  \dfrac{\sum_{i = 1}^n \exp(\lambda_i(X)/\varepsilon) q_i q_i ^\top}{\sum_{i = 1}^n \exp(\lambda_i(X)/\varepsilon)},
\end{equation*}
where $q_i$ is the $i^{\text{th}}$ column of the unitary matrix $Q$ in the eigen-decomposition $Q \Sigma Q^\top$ of $X$ with eigenvalues $\lambda_1\geq \cdots \geq \lambda_n$. In addition, $f_{\varepsilon}(X)$ is smooth with the smoothness parameter~$ 1/\varepsilon$. Figure~\ref{fig:maxcut_sdp} shows the simulation results of solving~\eqref{regularized maxcut dual} for different regularization coefficients~$\eta$ where the approximation level is~$\varepsilon = 10^{-5}$. The matrix $C$ is also generated using the Wishart distribution with $C = {G^\top G}/{\|G\|_{2}^2}$, where $G$ is a standard Gaussian matrix~\cite{helmberg2000spectral}, $n = 100$, and $\mathcal{R}(y) = \|y\|^2$.

\begin{figure}
    \centering
    \captionsetup{justification=centering}
    \includegraphics[scale=0.22]{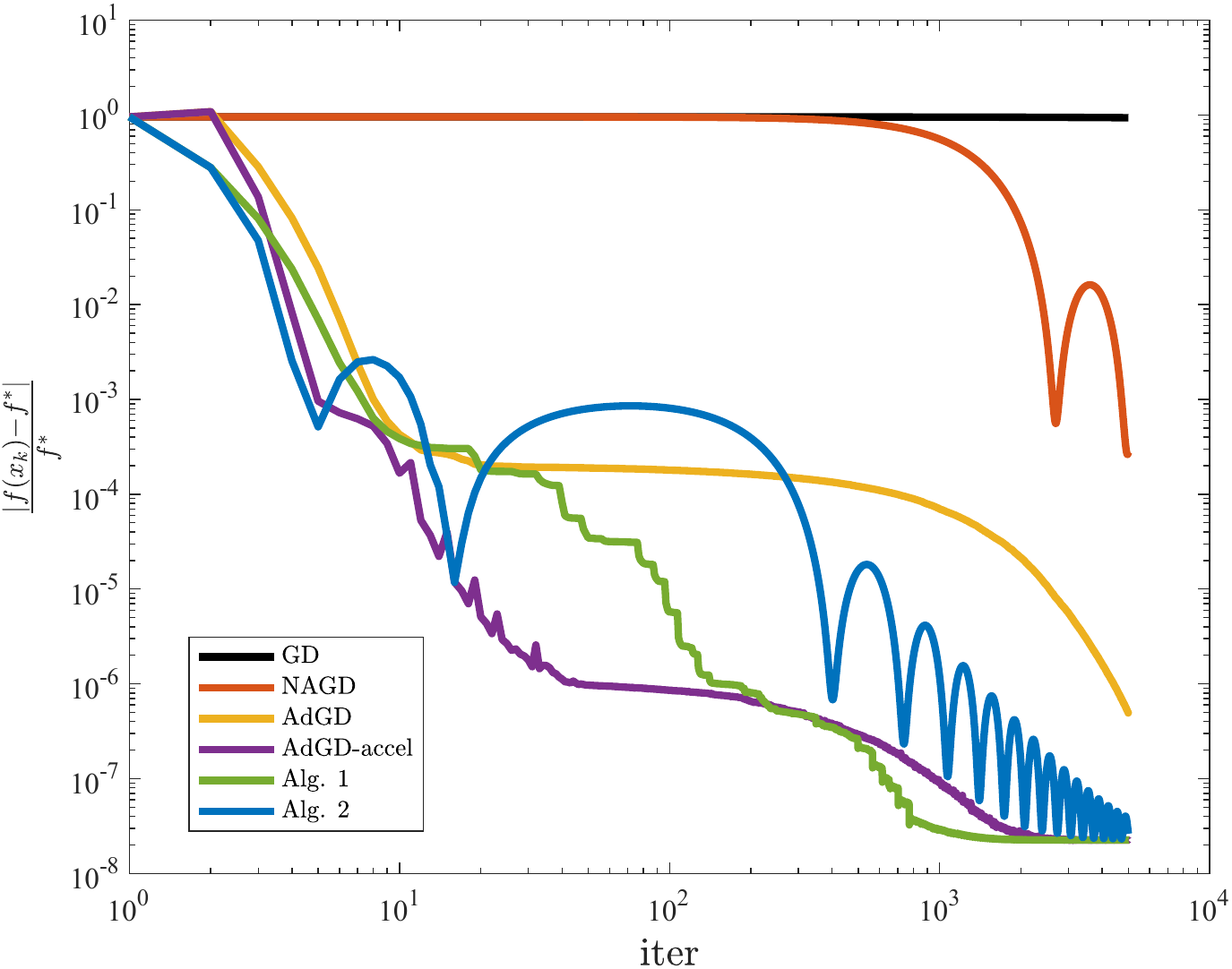}
    \includegraphics[scale=0.22]{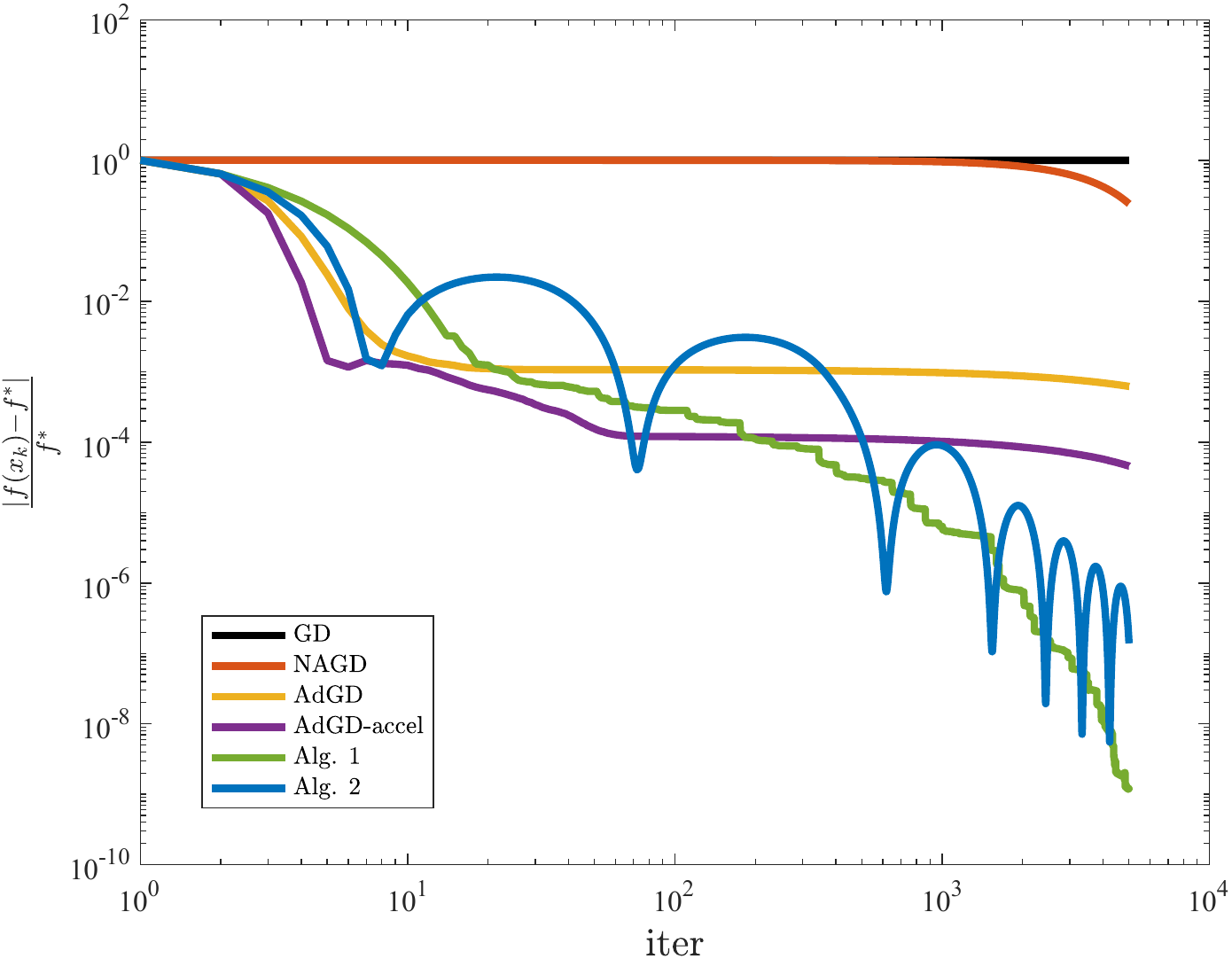}
    \includegraphics[scale=0.22]{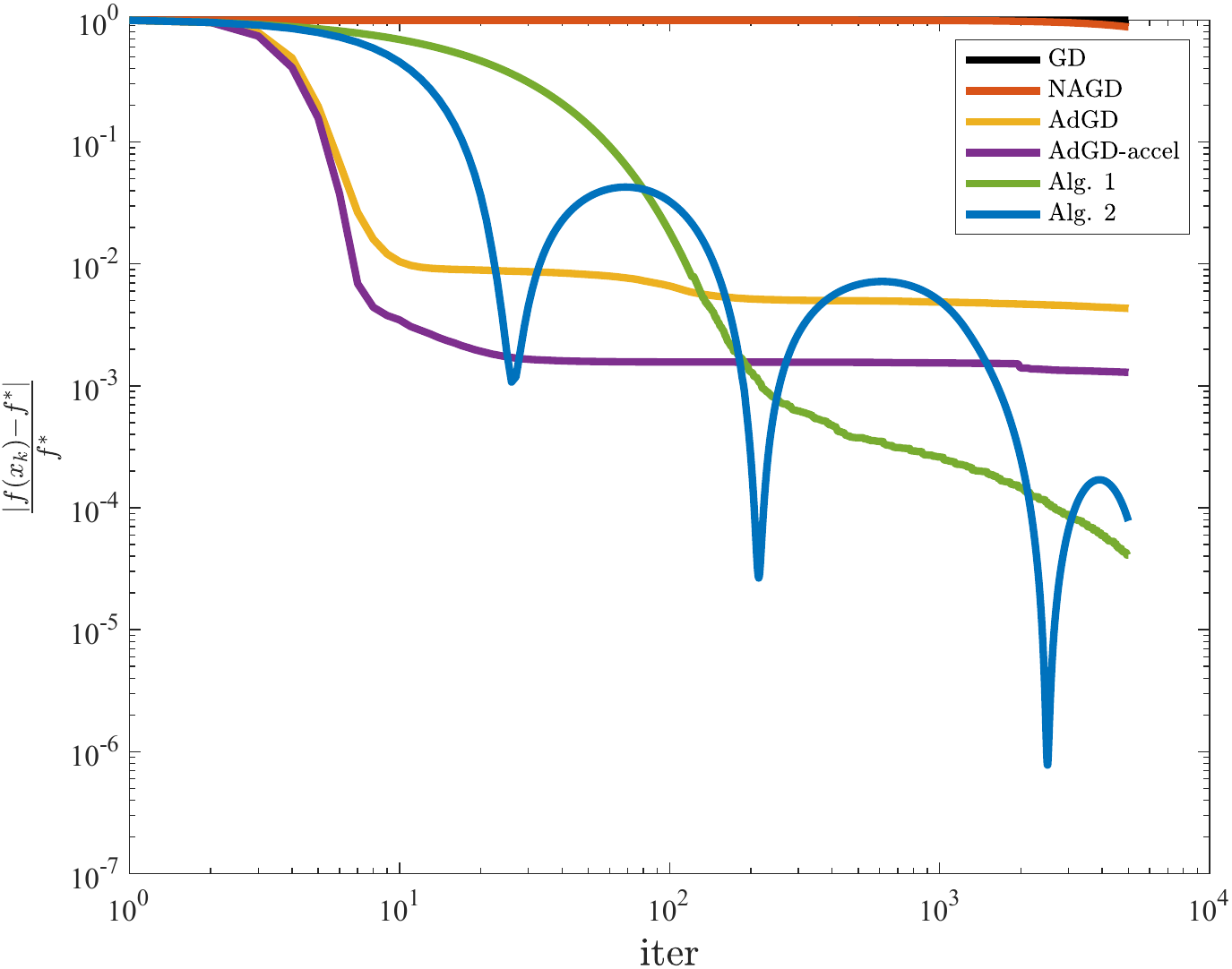}\\
    \subfloat[$\eta = 0.1$]{\label{fig:sdp1}\includegraphics[scale=0.22]{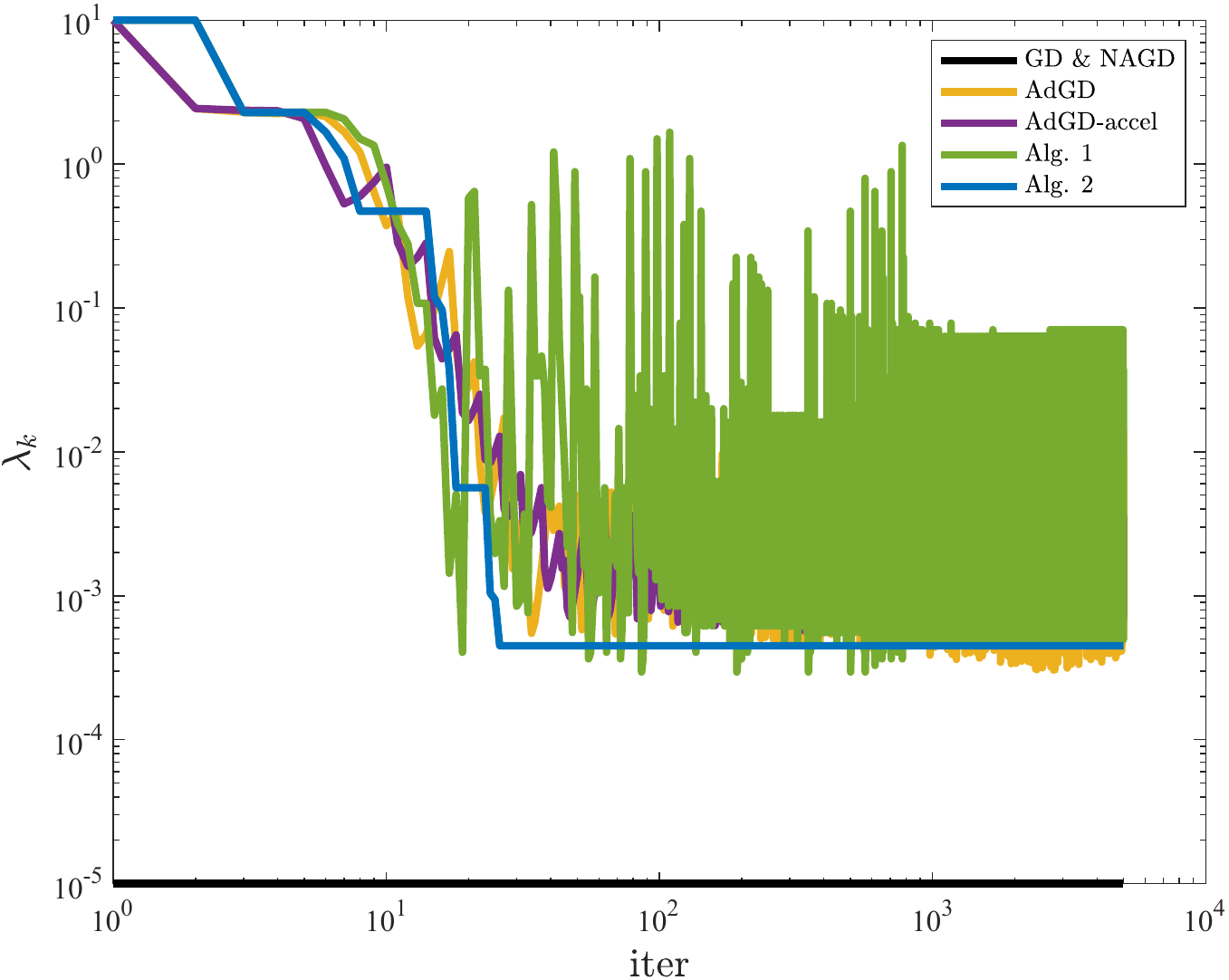}}
    \subfloat[$\eta = 0.01$]{\label{fig:sdp2}\includegraphics[scale=0.22]{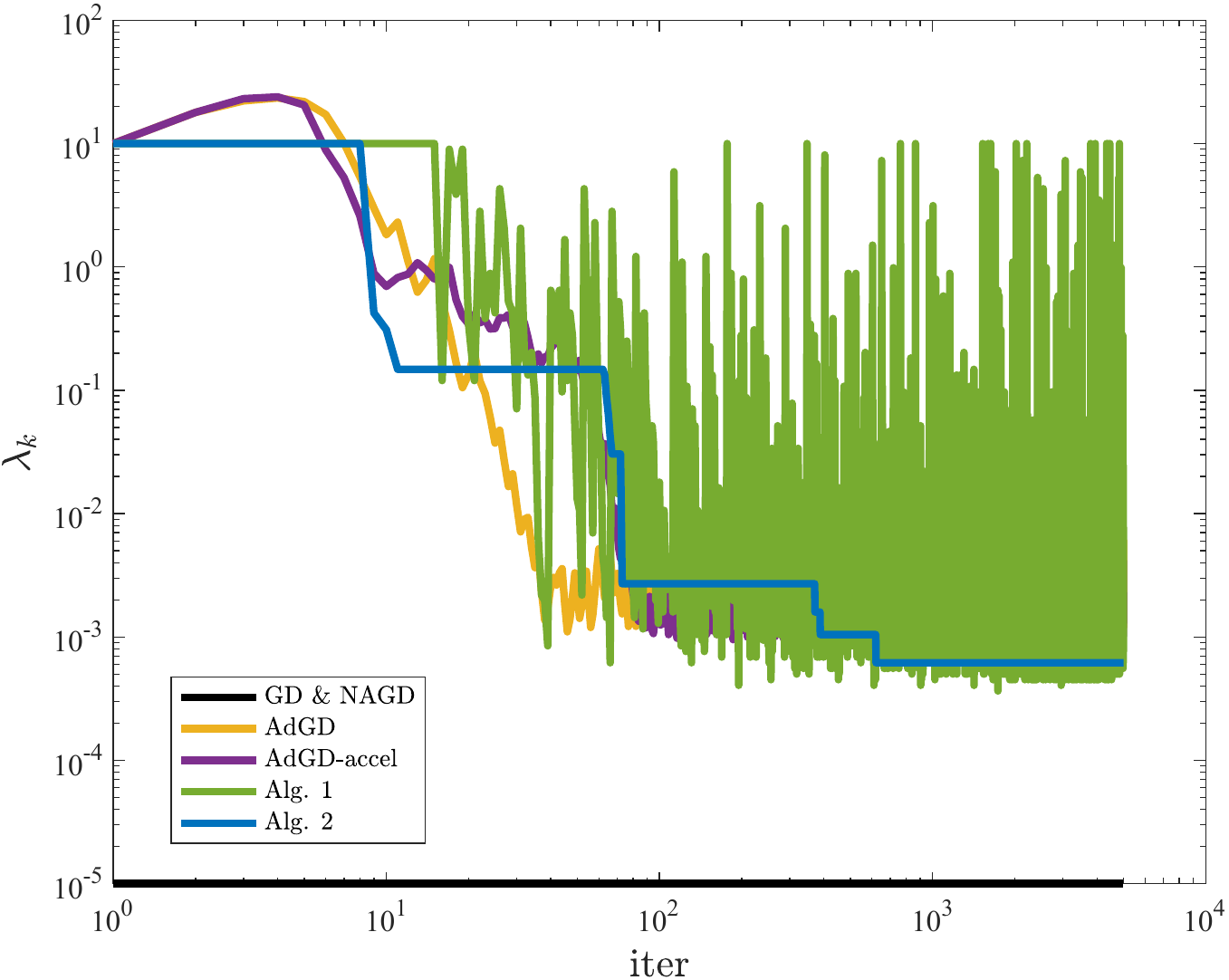}}
    \subfloat[$\eta = 0.001$]{\label{fig:sdp3}\includegraphics[scale=0.22]{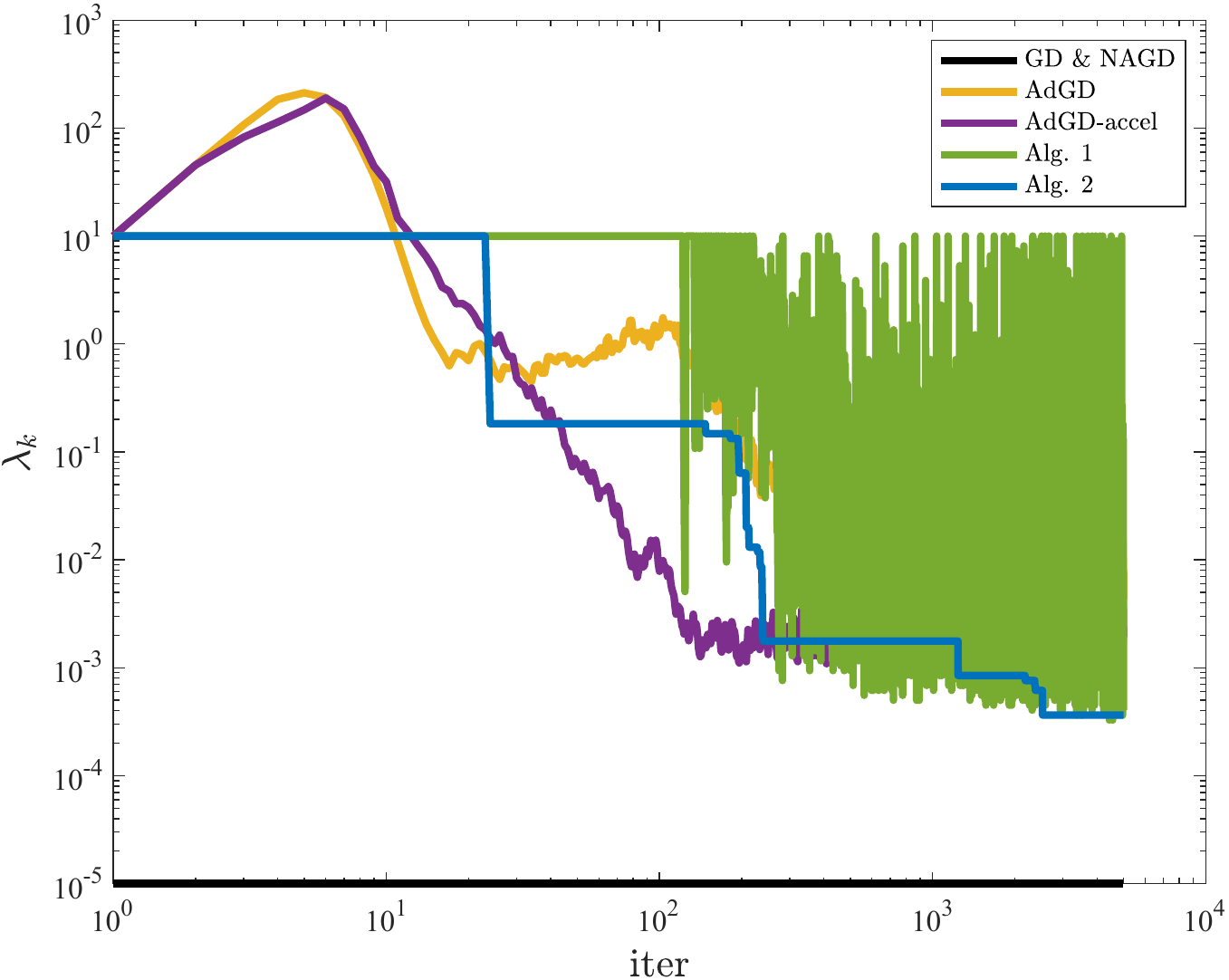}}
    \caption{Approximate maximum eigenvalue~\eqref{regularized maxcut dual}. The first row shows the optimality gap and the second row shows the stepsize behavior.}
    \label{fig:maxcut_sdp}
\end{figure}
\subsection{Composite minimization}
Motivated by sparse recovery and reconstruction problems \cite{neykov2016l1, selesnick2017sparse, gaines2018algorithms}, as well as sparse classification and feature selection \cite{bertsimas2021sparse}, we consider the following non-smooth composite minimization formulations:

\paragraph{\bf (i)~$\ell_1$-Regularized least~square~\cite{neykov2016l1, selesnick2017sparse}:}
\begin{align}
\label{L1-reg}
    \min\limits_{x \in \mathbb{R}^d} \underbrace{\|Ax-b\|_{2}^2}_{f(x)} + \underbrace{\gamma \|x\|_1}_{h(x)}.
\end{align}

\paragraph{\bf (ii) $\ell_1$-Constrained least~square~\cite{gaines2018algorithms}:} 
\begin{align}
\label{L1-const}
\min\limits_{x \in \mathbb{R}^d} \{\|Ax-b\|_{2}^2 \, : \|x\|_1 \leq 1\} = \min\limits_{x \in \mathbb{R}^d} \underbrace{\|Ax-b\|_{2}^2}_{f(x)} + \underbrace{\delta_{B_1[0, 1]}}_{h(x)}.
\end{align}

\paragraph{\bf (iii) $\ell_1$-Regularized logistic~regression~\cite{bertsimas2021sparse}:} 
\begin{align}
\label{L1-log-reg}
\min\limits_{x \in \mathbb{R}^d}  \underbrace{\dfrac{1}{N} \sum_{i=1}^{N} \log \left(1+\exp(-b_i a_{i}^\top x)\right)}_{f(x)} + \underbrace{\gamma \|x\|_1}_{h(x)}.
\end{align}

The functions $f(x)$ and $h(x)$ represent the smooth and prox-friendly parts, respectively. In the previous subsection, we explain the smooth parts and their smooth constant. The prox-friendly term~$h(x)$ has a closed form solution provided in Table~\ref{table2} where $\delta$, $B_1 [0,1]$, $[\cdot]_+$, and $\odot$ are indicator function, unit ball with 1-norm, positive part selector, and elementwise product, respectively~\cite{Beck2017}. 
\begin{table}[H]
    \centering
    \caption{Closed-form prox-operator of~$h(x)$.}
    \begin{tabular}{ccc}
    \hline
         {$h(x)$} && {$\text{prox}_{h(x)}$}\\
         \hline
         $\gamma\|x\|_1$ && $\tau_{\gamma}(x) = \big[|x| - \gamma \textbf{e}\big]_{+} \odot \textbf{sign}(x)$ \vspace{1mm} \\ 
         $\delta_{B_1[0, 1]}$ && $P_{B_1[0, 1]}(x) = \min\limits_{y \in {B_1[0, 1]}} \|y-x\|$\\
         \hline
    \end{tabular}
    \label{table2}
\end{table}

In our simulations, we set $N = d = 200$, $\gamma$ proportional to ${1}/{N}$, $A = 5\cdot\texttt{rand(n,n)}$, and $b=\texttt{rand(n,1)}$, where \texttt{rand($\cdot $)} is uniformly distributed random numbers in the interval $(0,1)$. The results are provided in Figure~\ref{fig:img8}. We wish to note that for the objective functions~\eqref{L1-reg} and \eqref{L1-const}, the proposed accelerated stepsize~\ref{alg:Algorithm2} has a competitive result and similar behavior with the best performance FISTA~\cite{Beck2009} while in the case of \eqref{L1-log-reg} it outperforms all the other algorithms notably. 

\begin{figure}[H]
    \centering
    \captionsetup{justification=centering}
    \includegraphics[scale=0.295]{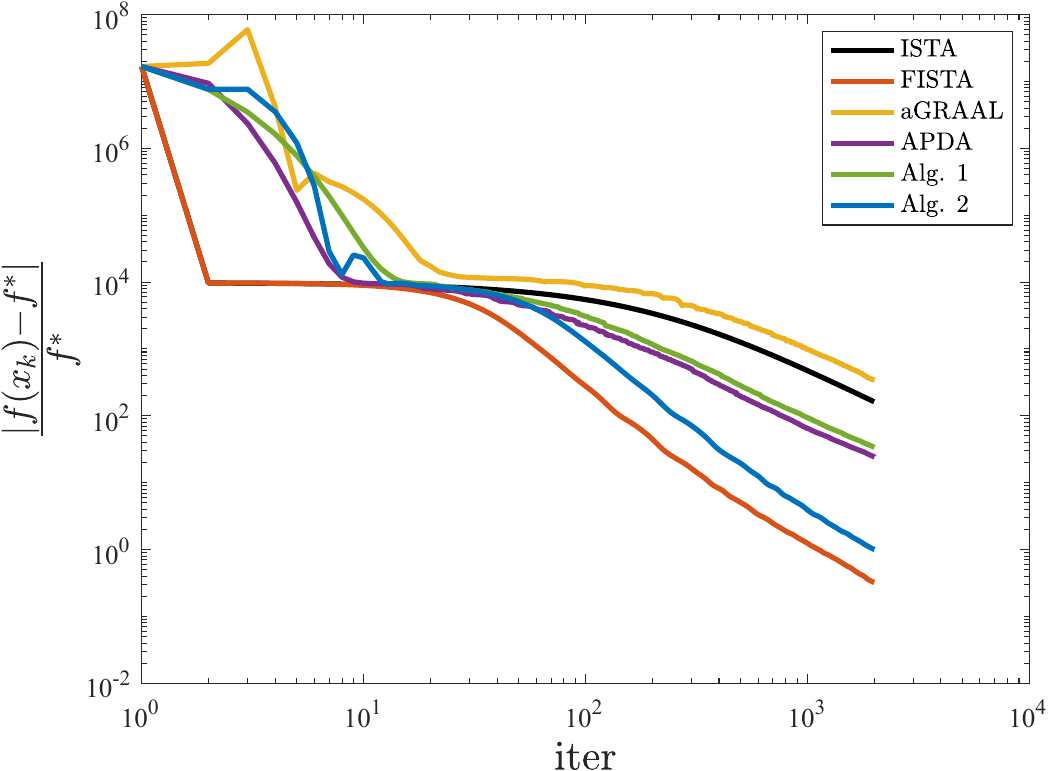}
    \includegraphics[scale=0.295]{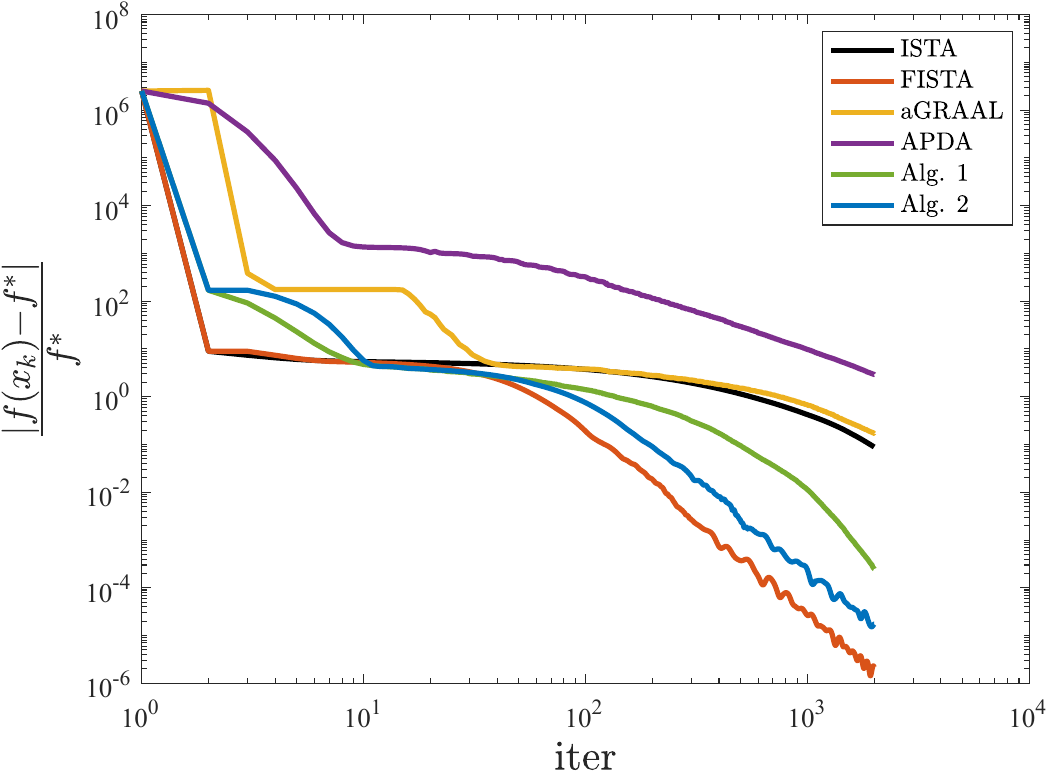}
    \includegraphics[scale=0.295]{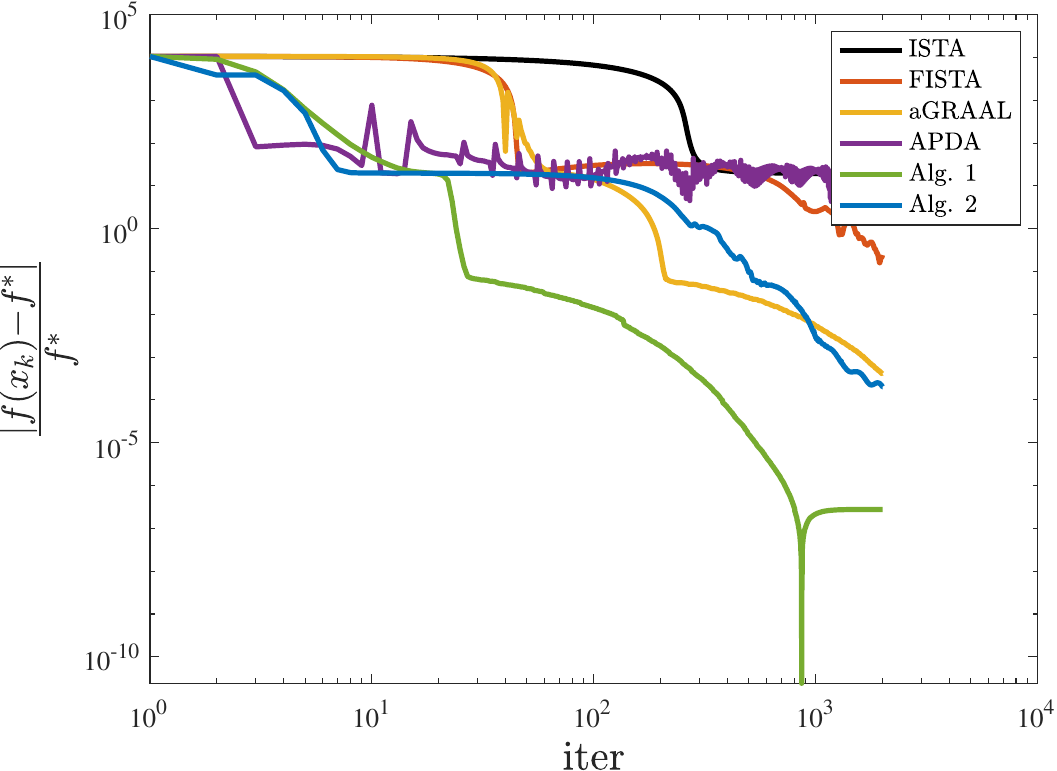}\\
    \subfloat[$\ell_1$-Regularized least squares \eqref{L1-reg}]{\label{fig8:s2fig1}\includegraphics[scale=0.312]{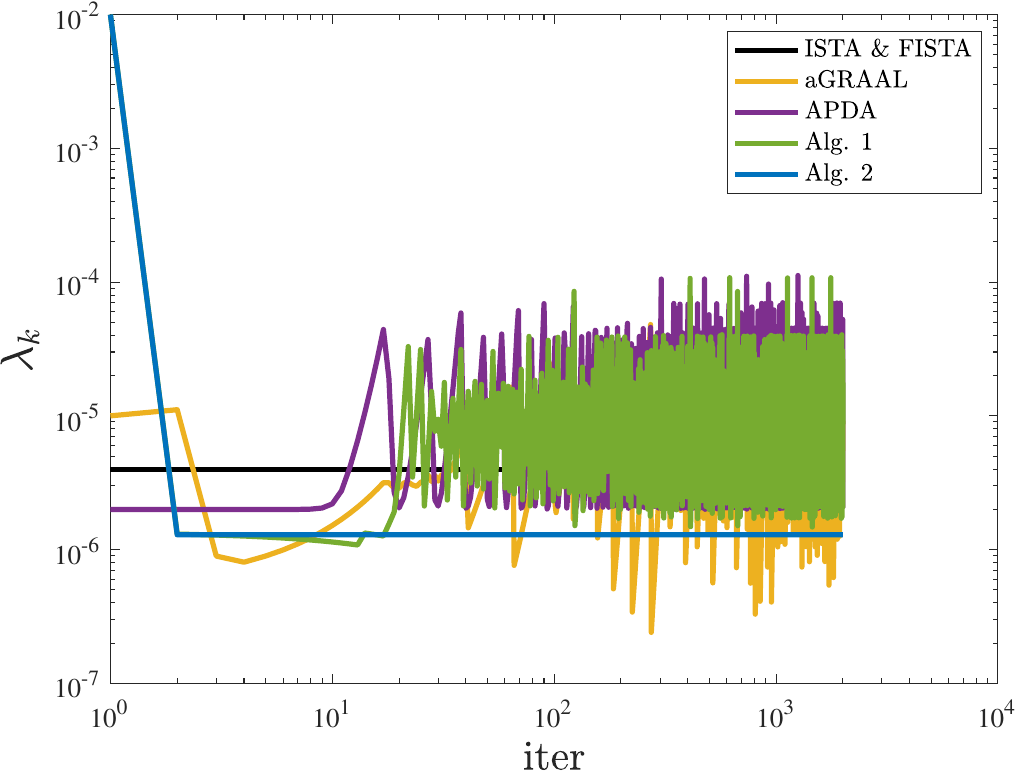}}
    \subfloat[$\ell_1$-Constrained least squares \eqref{L1-const}]{\label{fig8:s2fig2}\includegraphics[scale=0.312]{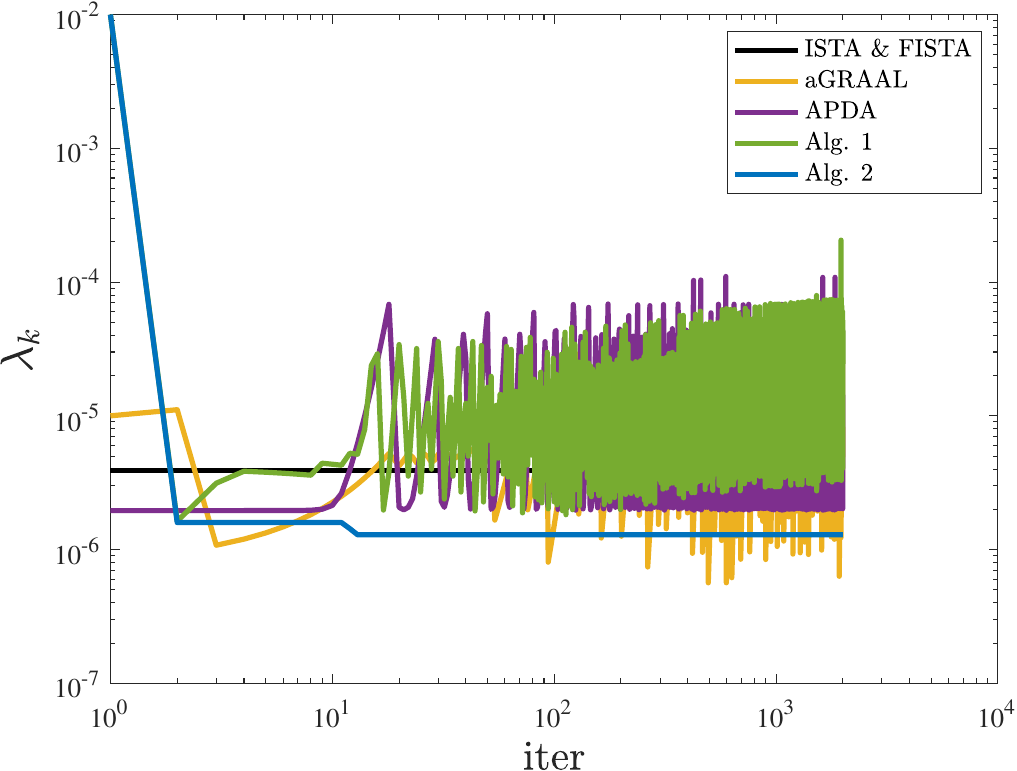}}
    \subfloat[$\ell_1$-Regularized logistic reg~\eqref{L1-log-reg}]{\label{fig8:s2fig3}\includegraphics[scale=0.312]{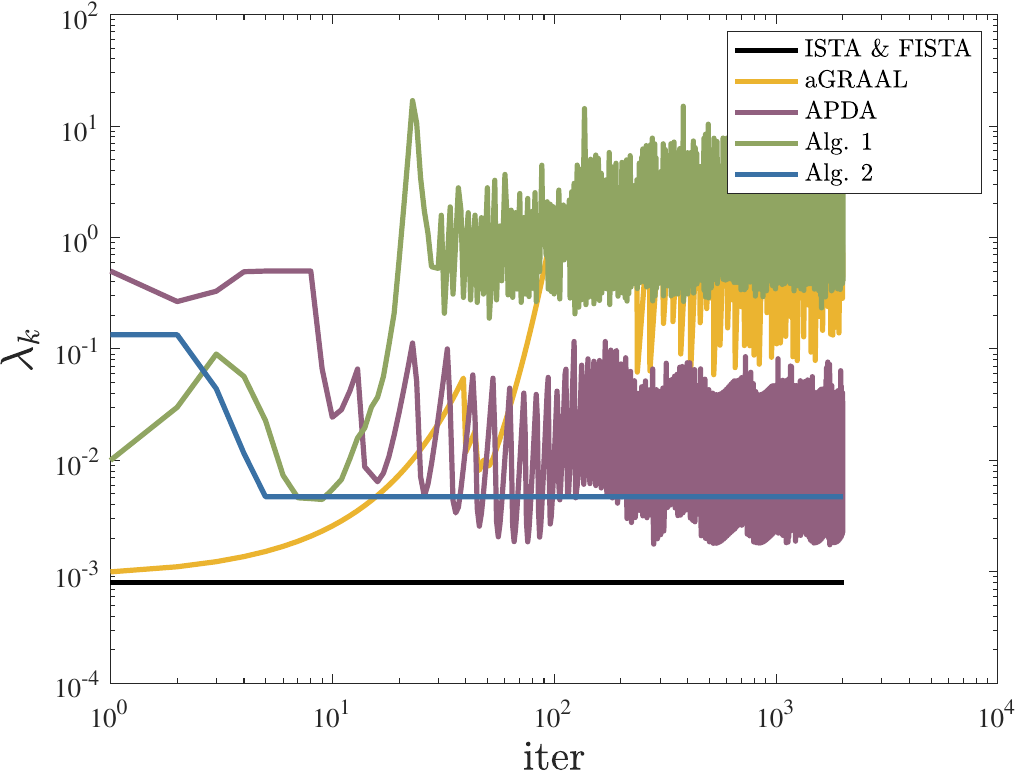}}
    \caption{The results for the class (2)~composite minimization. The first row shows the optimality gap, and the second row shows the stepsize behavior.}
    \label{fig:img8}
\end{figure}
\begin{figure}[H]
    \centering
    \captionsetup{justification=centering}
    \includegraphics[scale=0.295]{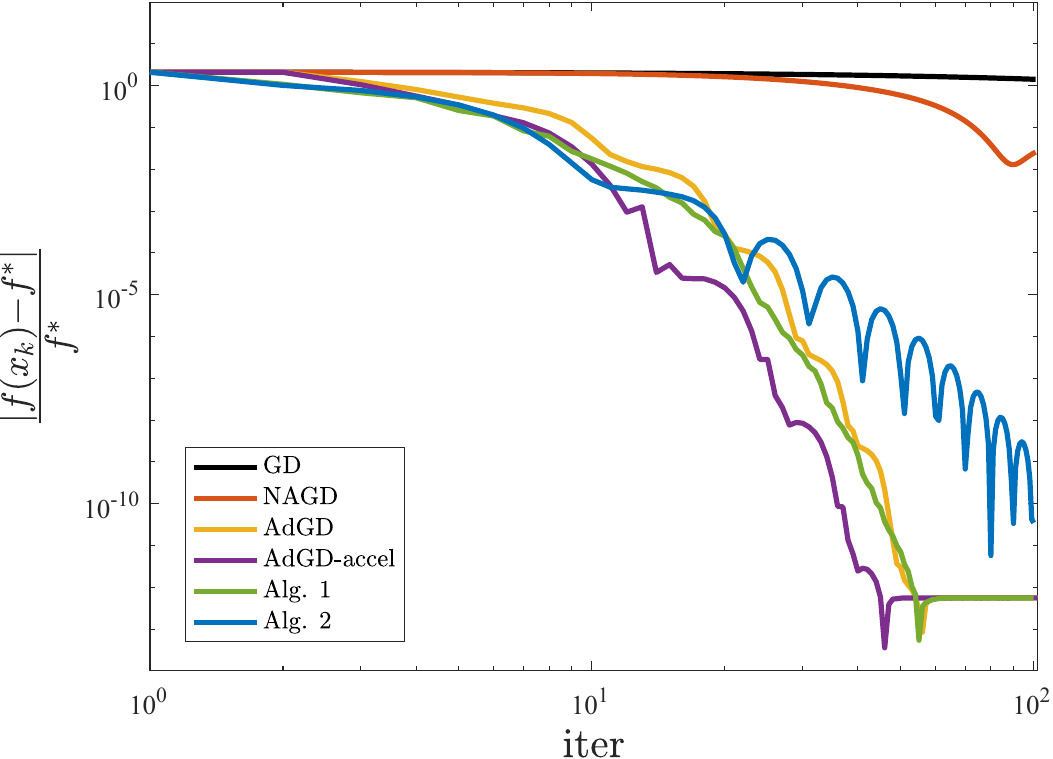}
    \includegraphics[scale=0.295]{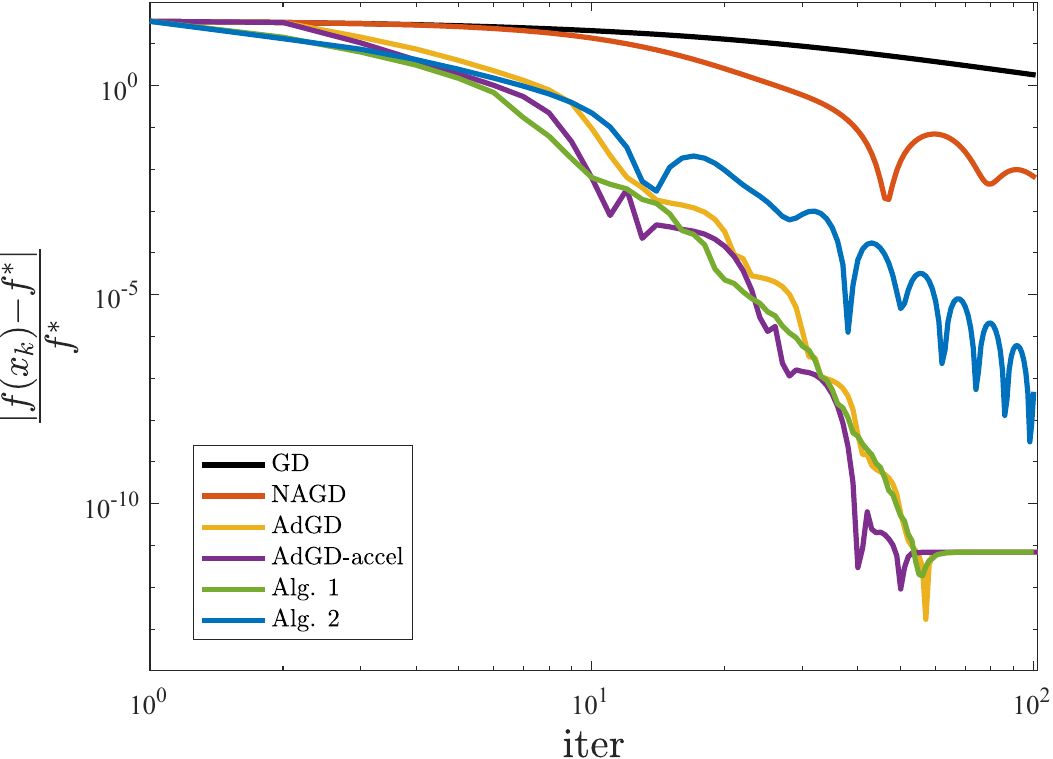}
    \includegraphics[scale=0.295]{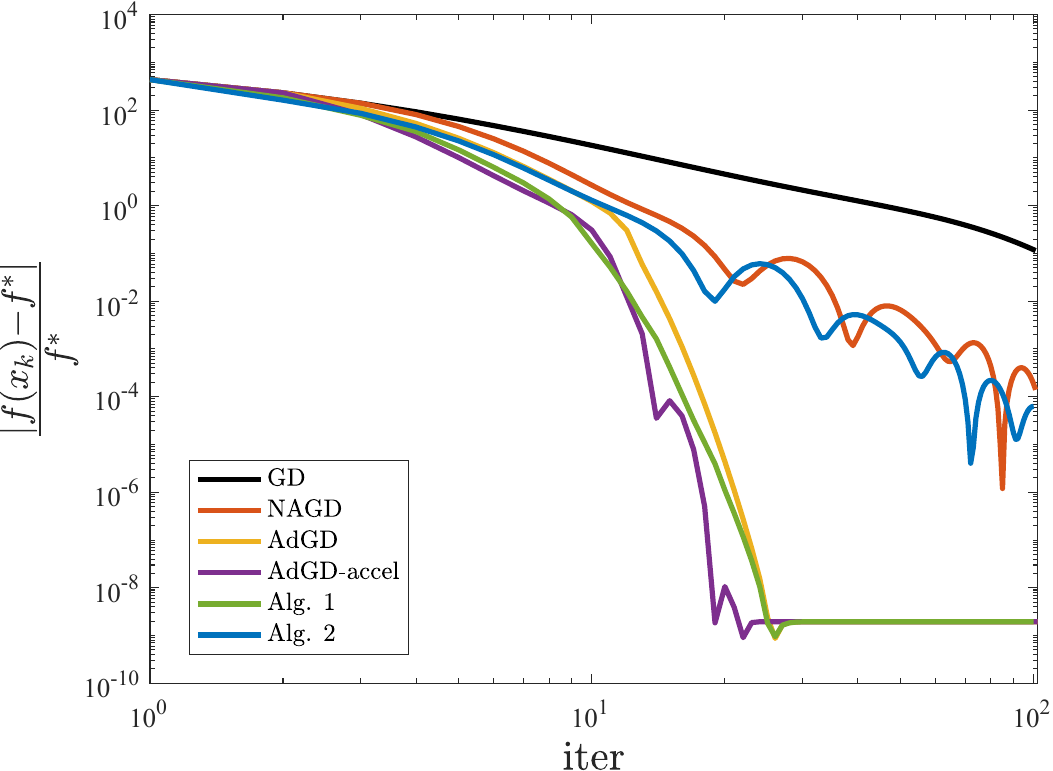}\\
    \subfloat[$M=0.5$]{\label{fig7:s2fig1}\includegraphics[scale=0.305]{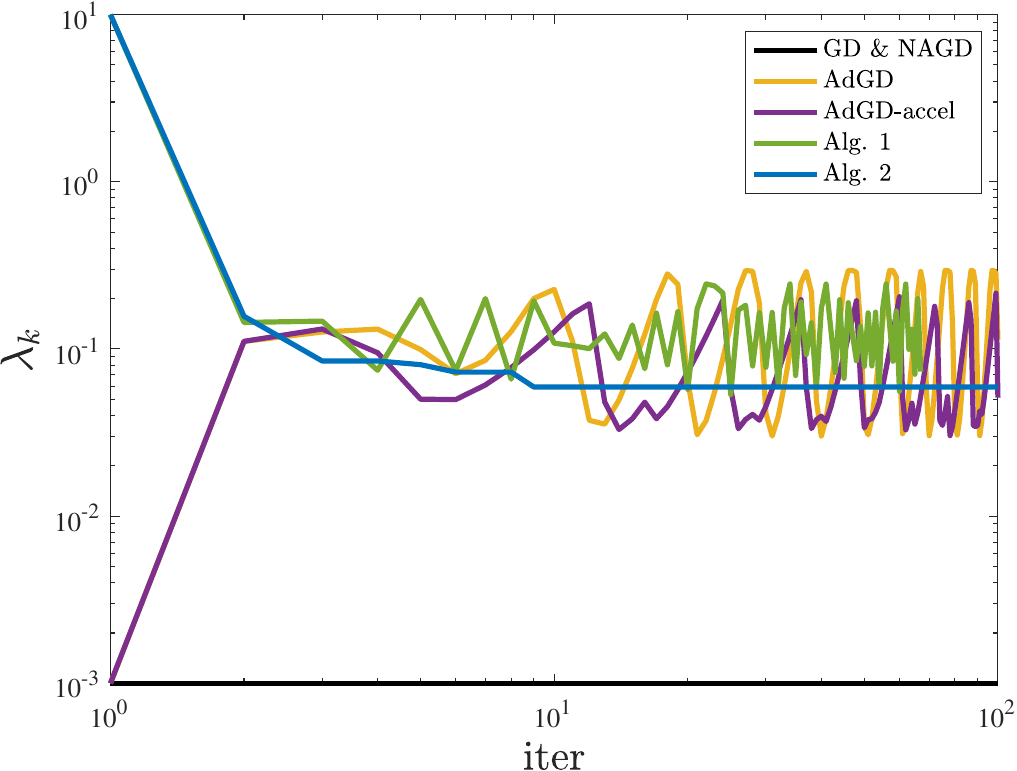}}
    \subfloat[$M=5$]{\label{fig7:s2fig2}\includegraphics[scale=0.305]{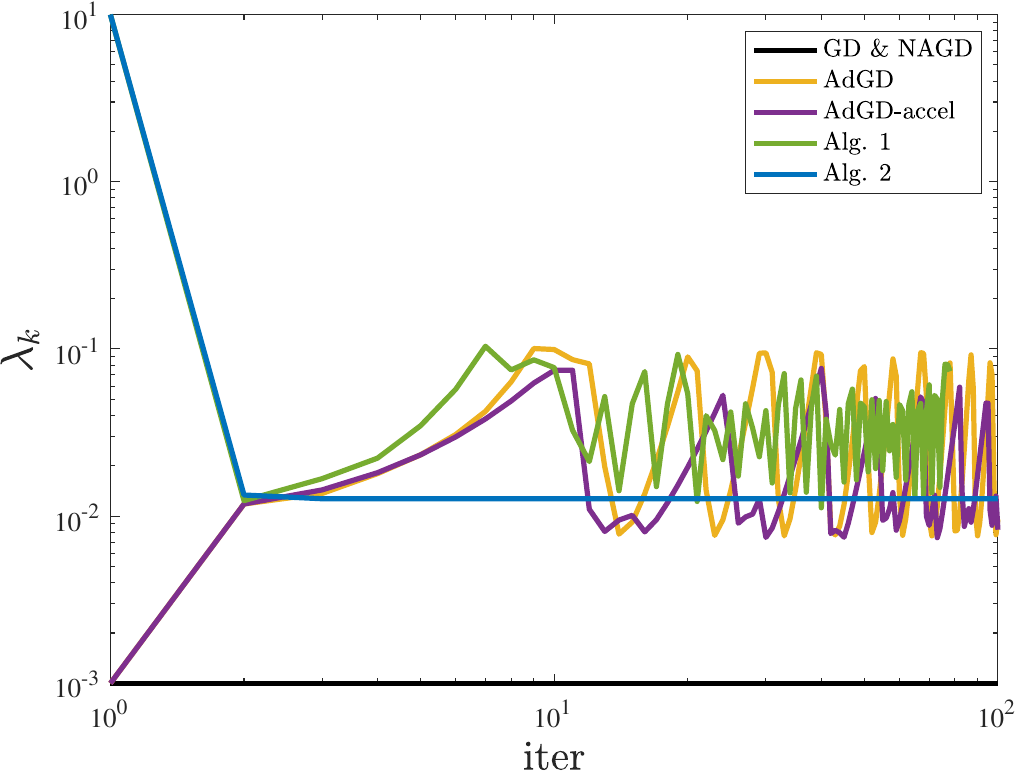}}
    \subfloat[$M=50$]{\label{fig7:s2fig3}\includegraphics[scale=0.305]{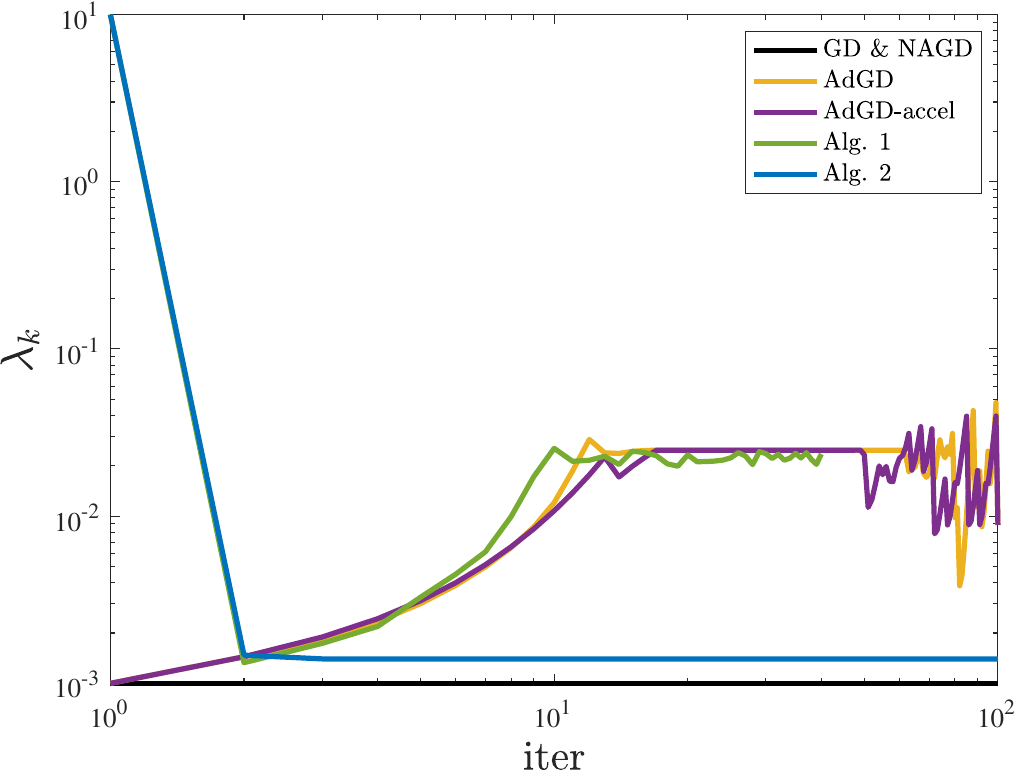}}
    \caption{Cubic regularization problem~\eqref{cubic}. The first row shows the optimality gap, and the second row shows the stepsize behavior for different values of~$M$.}
    \label{fig:img7}
\end{figure}
\subsection{Non-convex minimization:}
We also test our algorithm in non-convex minimization. We consider the case of cubic optimization that is useful in high-order methods and trust-region subproblems \cite{Nesterov2006, chen2022accelerating} which is defined as 
\begin{equation}
\label{cubic}
    \min_{x\in\mathbb{R}^d} \dfrac{1}{2}x^\top H x + g^\top x + \dfrac{M}{6} \|x\|^3, 
\end{equation}
where the matrix~$H\in \mathbb{R}^{d\times d}$, vector $g\in \mathbb{R}^d$, and scalar $M > 0$ are the problem data. Note that in higher-order methods, the constant $M$ has an impact on the convergence rate~\cite{Nesterov2008}. Due to the cubic term, the function~\eqref{cubic} is neither convex nor smooth. The parameters~$g$ and $H$ are the gradient and the Hessian of the logistic loss~\eqref{log-reg}, respectively, and are computed using the same data as in Subsection~\ref{subsec: smooth-numerics} ((i)~logistic regression). We provide the simulations for different $M$ with $d = N = 200$, and the results are reported in Figure~\ref{fig:img7}.
{\subsection{Real-world applications and comparison with Armijo rule:}
Indeed, the applicability of the proposed method is an important consideration. On the other hand, comparing our proposed method with the classical Armijo linesearch rule, the closest in spirit to our approach, motivated us to include 12 diverse numerical examples of varying dimensionality in this section, including comparisons with the Armijo linesearch as well as other methods.\\
We include simulations for logistic regression~\eqref{log-reg} (Subsection 4.1) and composite objective functions (Subsection 4.2), which are widely used across various domains, including image processing, image reconstruction, machine learning, and control. The simulations are conducted using real-world datasets from the LIBSVM dataset collection~\cite{chang2011libsvm}, the application domains of which are summarized in Table \ref{libsvm_dataset}.\\
Figures~\ref{classification-prob} and \ref{regression-prob} report the results for the classification (($\ell_1$-regularized) logistic regression~\eqref{log-reg}--\eqref{L1-log-reg}) and regression problems ($\ell_1$-regularized/constrained least squares \eqref{L1-reg}--\eqref{L1-const}), respectively, where all algorithms are initialized at the same randomly chosen point. As shown in these figures, the proposed methods compare favorably to the alternatives, particularly the Armijo line search method.}

\begin{table}[H]
\centering
\caption{Summary of selected LIBSVM datasets and their applications~\cite{chang2011libsvm}.}
\hspace{-2cm}
\small
\makebox[\textwidth][l]{%
\begin{tabular}{c c c}\hline
{Dataset} & {Application} & {Task type}\\
\hline & \\[-1.0em]
\texttt{a1a} & Income classification based on demographic features (Adult dataset variant) & Classification (Logistic)\\
\texttt{w4a} & Web-based high-dimensional classification derived from Adult dataset & Classification (Logistic)\\
\texttt{breastcancer} & Tumor classification (benign vs malignant) using medical features & Classification (Logistic) \\
\texttt{abalone} & Predict age from physical shell measurements (marine biology study) & Regression (Squared Loss) \\
\texttt{smallcpu} & Predict CPU performance from hardware specifications & Regression (Squared Loss) \\
\texttt{triazines} & Predict biological activity of chemical compounds (QSAR modeling) & Regression (Squared Loss)\\
\\[-1.0em]
\hline
\end{tabular}}
\label{libsvm_dataset}
\end{table}
\begin{figure}
    \centering
    \captionsetup{justification=centering}
    \includegraphics[scale=0.47]{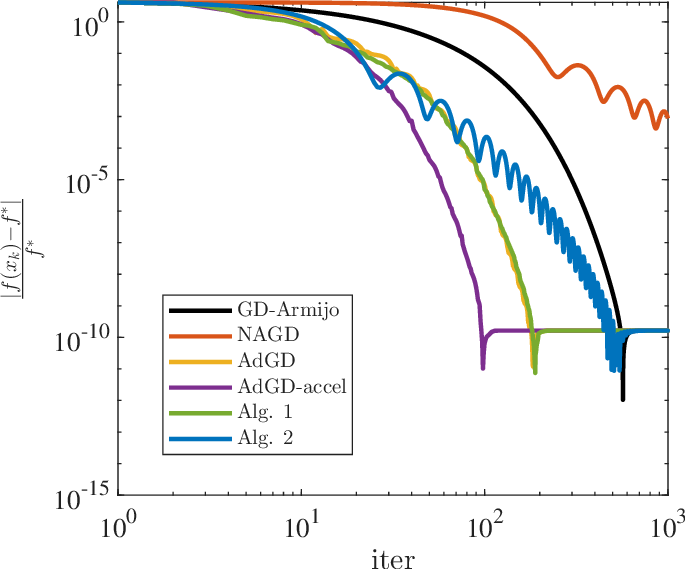}
    \includegraphics[scale=0.47]{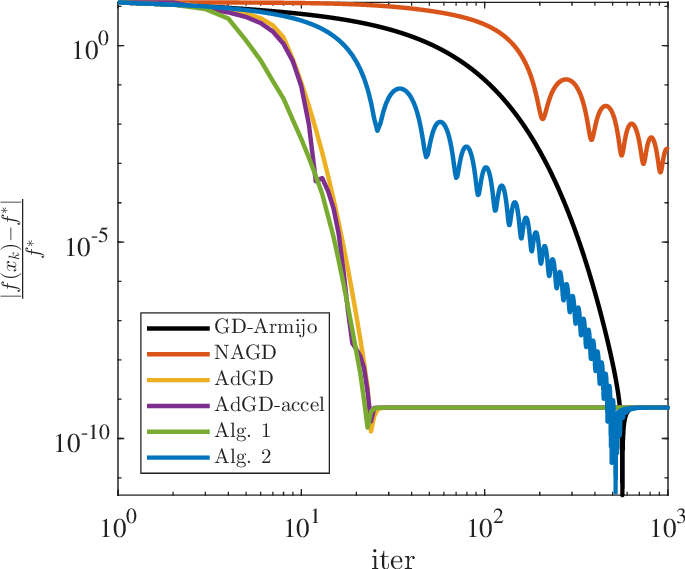}
    \includegraphics[scale=0.47]{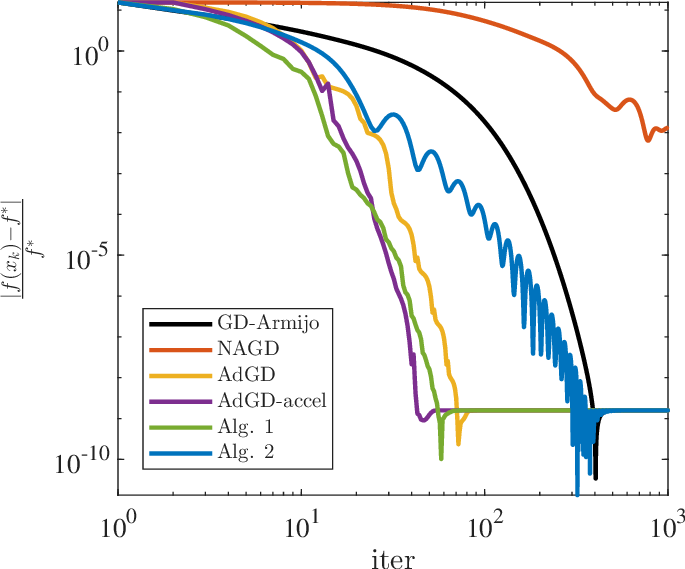}\\
    \subfloat[\texttt{a1a}]{\includegraphics[scale=0.47]{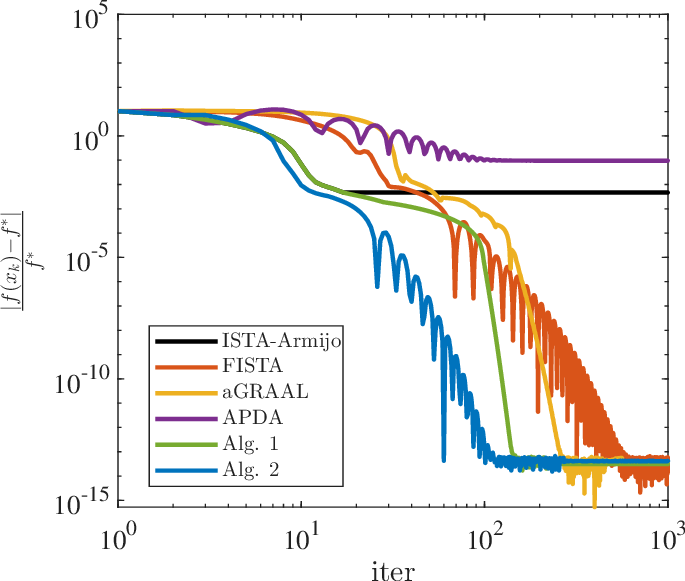}}
    \subfloat[\texttt{w4a}]{\includegraphics[scale=0.47]{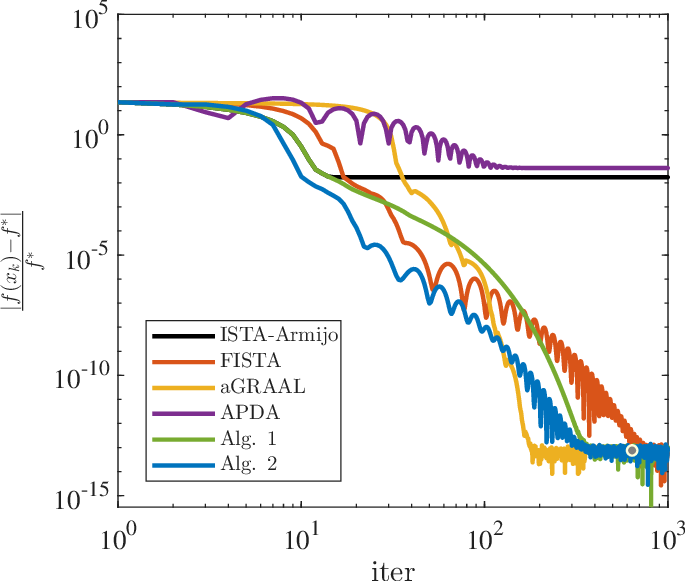}}
    \subfloat[\texttt{breastcancer}]{\includegraphics[scale=0.47]{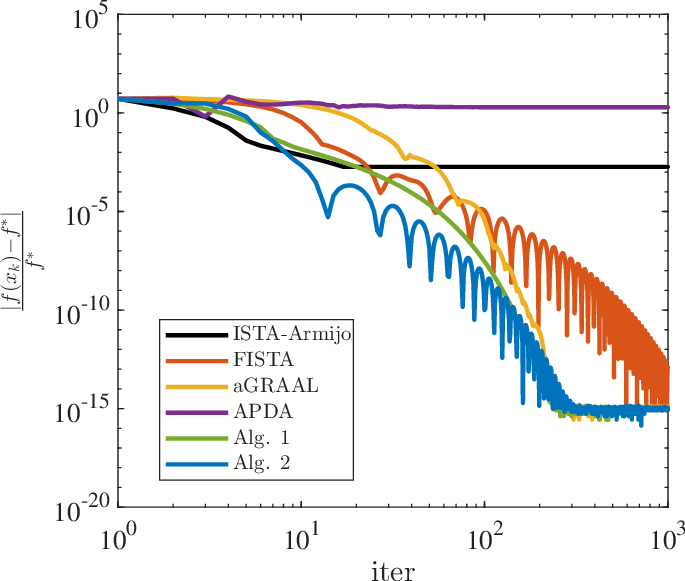}}
    \caption{{Classification problem. The first and second rows show the optimality gap for logistic regression~\eqref{log-reg} and $\ell_1$-regularized logistic regression~\eqref{L1-log-reg}, respectively.}}
    \label{classification-prob}
\end{figure}
\begin{figure}[H]
    \centering
    \captionsetup{justification=centering}
    \includegraphics[scale=0.47]{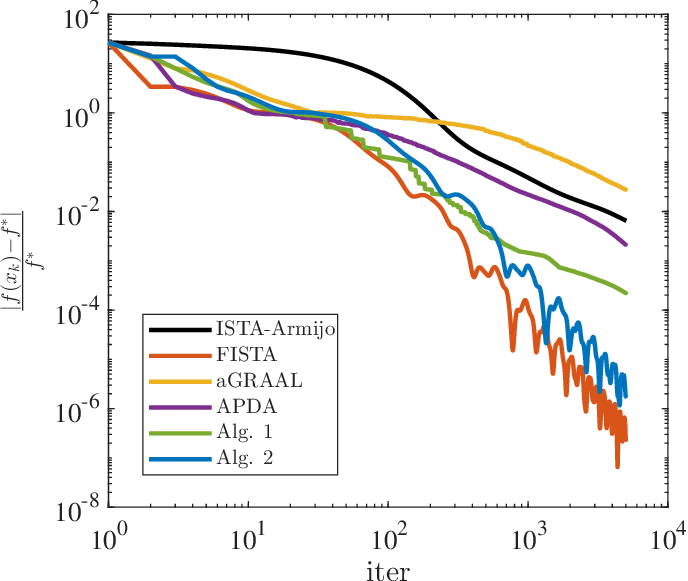}
    \includegraphics[scale=0.47]{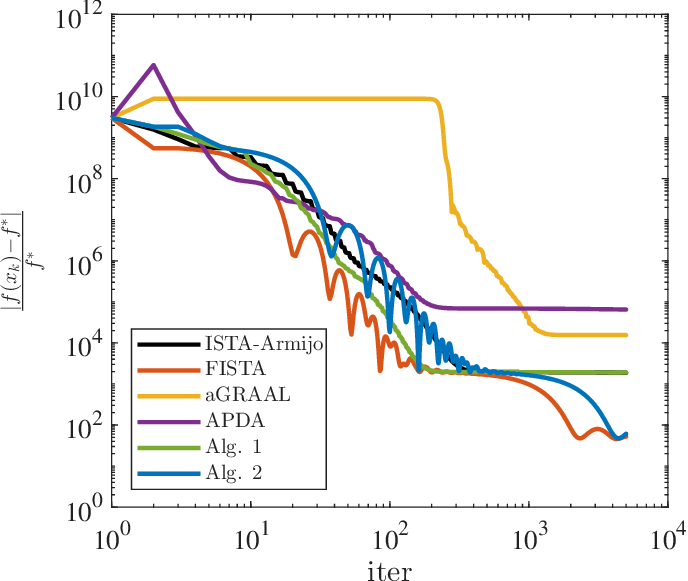}
    \includegraphics[scale=0.47]{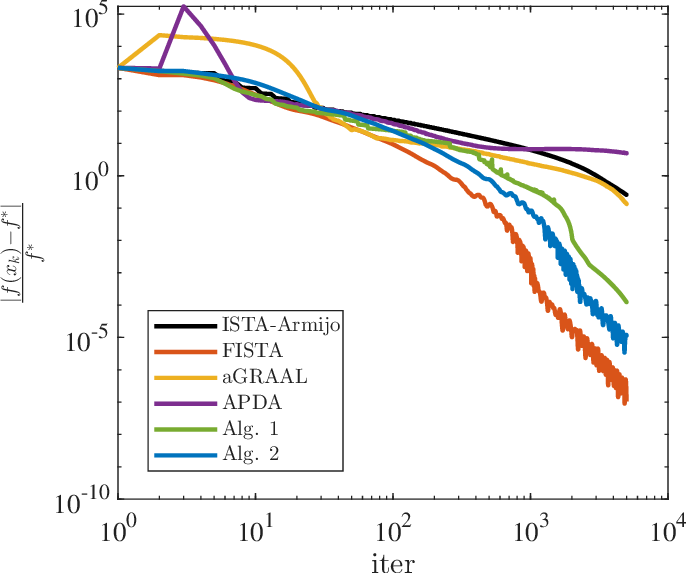}\\
    \subfloat[\texttt{abalone}]{\includegraphics[scale=0.47]{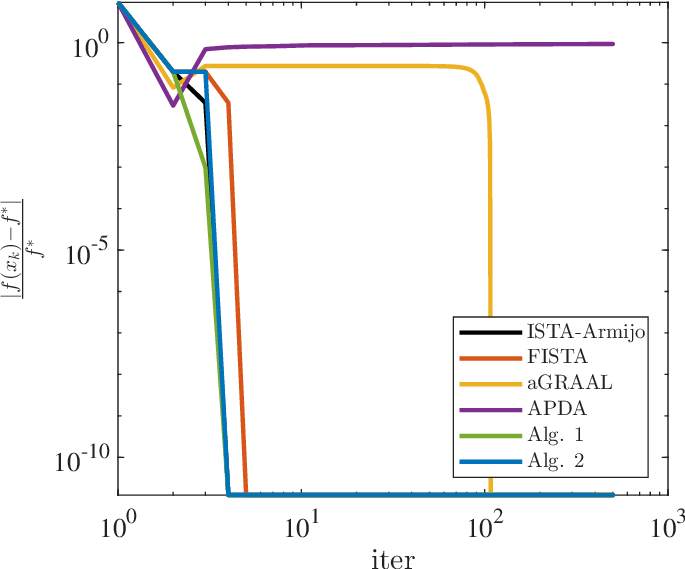}}
    \subfloat[\texttt{smallcpu}]{\includegraphics[scale=0.47]{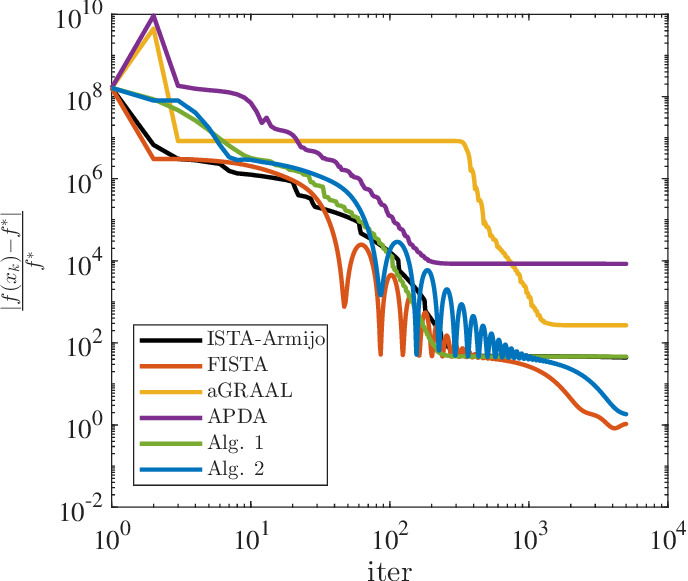}}
    \subfloat[\texttt{triazines}]{\includegraphics[scale=0.47]{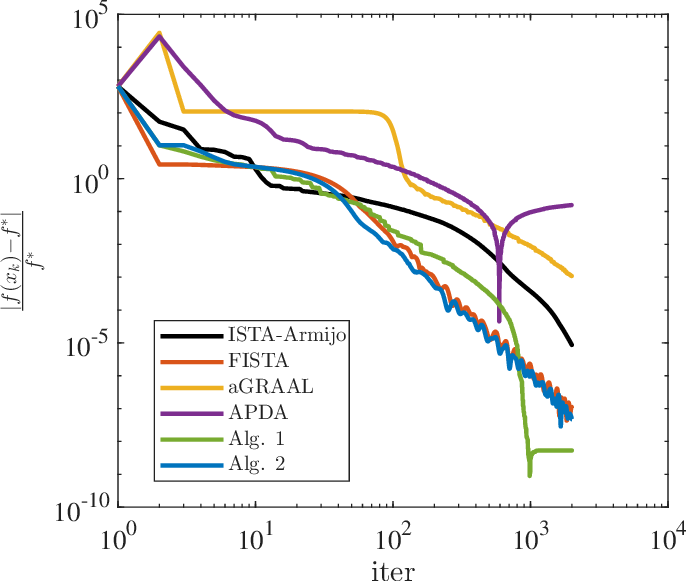}}
    \caption{{Regression problem. The first and second rows show the optimality gap for $\ell_1$-regularized~\eqref{L1-reg}/constrained~\eqref{L1-const} least squares, respectively.}}
    \label{regression-prob}
\end{figure}
\section{Conclusion and Future Directions} \label{conclussion}
This research proposes a new stepsize rule for non-accelerated and accelerated first-order methods for convex composite optimization. The proposed approach provably achieves the convergence optimal rate of $\mathcal{O}(k^{-2})$ for accelerated methods. Compared to other stepsize rules, our proposed scheme is implicit but only requires zeroth-order information to compute a suitable stepsize. Possible extensions to this work include the $min\,max$ problem \cite{mokhtari2020convergence,yoon2021accelerated, boct2022accelerated}, stochastic gradient descent optimization \cite{lacoste2012simpler}, variational inequalities \cite{nemirovski2004prox,juditsky2011solving, malitsky2020forward}, and non-convex optimization \cite{jin2018accelerated, vladarean2021first, alacaoglu2021convergence}. To mitigate the fact that the stepsize in \ref{alg:Algorithm2} is non-increasing, one could re-initialize it finitely many times after a predefined number of iterations. This technique is similar to the restarted accelerated gradient method \cite{Emmanuel2015,Kim2018, yang2018rsg} and it is left as future work.

The authors of a recent work \cite{liang2017activity} demonstrate that iterations generated by Forward–Backward splitting methods, which include several variants (e.g., ISTA, FISTA), lie on active manifolds within a finite number of iterations and enter a regime of local linear convergence. However, their analysis is based on the standard assumption that the stepsize belongs to $[0, 2/L]$. Departing from this convention is one of the future areas of interest. The authors in \cite{liang2017activity} also establish and explain why FISTA locally oscillates and can be slower than ISTA under the same standard assumption $\lambda \in [0,2/L]$. We observe similar behavior in our simulation results (Figures \ref{fig:img4}, \ref{fig8:s2fig2}, \ref{fig:img7}). Analyzing this behavior for $\lambda > 2/L$, which can also occur in our proposed stepsize rules, could be another future direction.
\bibliographystyle{plain}
\bibliography{references.bib}
\end{document}